\documentclass[10pt]{amsart}
\usepackage{amsmath, amsthm, amssymb}
\usepackage{amsfonts}
\usepackage{mathrsfs}
\usepackage[all]{xy}
\usepackage{wasysym}
\usepackage{bbm}
\usepackage{textcomp}
\usepackage{verbatim}
\usepackage{charter}
\usepackage{fancyhdr}
\usepackage{epsfig}
\usepackage{pb-diagram,pb-xy}      
\usepackage{pstricks}               
\usepackage{pst-all}                
\usepackage{hyperref}
\hypersetup{
    colorlinks,
    citecolor=blue,
    filecolor=blue,
    linkcolor=blue,
    urlcolor=blue
}

\theoremstyle{plain}
\newtheorem{theorem}{Theorem}[section]
\newtheorem{lemma}[theorem]{Lemma}
\newtheorem{proposition}[theorem]{Proposition}

\theoremstyle{definition}
\newtheorem{definition}[theorem]{Definition}
\newtheorem{remark}[theorem]{Remark}

\newcommand{\di}{\mathrm{d}}

\begin{document}

\title{Poisson Cohomology of Broken Lefschetz Fibrations}

\author{Panagiotis Batakidis}

\address{Department of Mathematics 
\\
Aristotle University of Thessaloniki
\\
 Thessaloniki 54124, Greece}

\email{batakidis@math.auth.gr, batakidis@gmail.com}

\author{Ram{\'o}n Vera}

\address{Department of Mathematics, KU Leuven, Celestijnenlaan 200B, Leuven 
B-3001, Belgium}

\email{ramon.vera@kuleuven.be, rvera.math@gmail.com}

\begin{abstract}
We compute the formal Poisson cohomology of a broken Lefschetz fibration by calculating it at fold and Lefschetz singularities. Near a fold singularity the computation reduces to that for a point singularity in 3 dimensions. For the Poisson cohomology around singular points we adapt techniques developed for the Sklyanin algebra.  As a side result, we give compact formulas for the Poisson coboundary operator of an arbitrary Jacobian Poisson structure in 4 dimensions.

\end{abstract}


\maketitle

\vspace{-1cm}

\section{Introduction}

Poisson geometry originated in the Hamiltonian formalism of classical mechanics, and plays an important role in the passage to quantum mechanics. A primary tool in studying invariants of Poisson manifolds is Poisson cohomology. It reveals important features of the geometry of a Poisson manifold such as the modular class, obstructions to deformations, normal forms and deformation quantization.  However, calculating Poisson cohomology can be very difficult.  The cohomology spaces are generally infinite-dimensional, and they are unknown for many Poisson structures as there is no general method for their computation. One has better chances when restricting to formal coefficients for the cochain complex of multivector fields.

For weight homogeneous Poisson algebras $\mathcal{A}=\mathbb{R}[[x_1,x_2,x_3]]$ in 3 variables, Pichereau \cite{P06} computed the Poisson (co)homology with formal coefficients under the assumption that the structure is determined by a weight homogeneous polynomial $\phi$ with isolated singularity. Note that in this case the quotient space $\mathcal{A}/\langle \partial_1\phi,\partial_2\phi,\partial_3\phi\rangle$ has finite Milnor number. Following this work, Pelap \cite{P09} gave formulas for the formal Poisson (co)homology of the Sklyanin algebra. The structure there is unimodular and weight homogeneous, with two weight homogeneous Casimirs forming a complete intersection with isolated singularity.  In dimension 4, Hong and Xu computed the Poisson cohomology of del Pezzo surfaces \cite{HX11}. Under some conditions, Monnier \cite{M02}  computed the formal Poisson cohomology of quadratic structures.

Broken Lefschetz fibrations (bLfs) originated as a generalization of Lefschetz pencils \cite{D99, ADK05}, and in recent years they have found diverse applications in low-dimensional topology, symplectic geometry, and singularity theory \cite{BH15, GK15, In16, Wi11}.  A bLf is a map from a 4--manifold $M$ to the 2-sphere, with a singularity set consisting of a finite collection of circles which can be assumed to be disjoint, called \textsl{fold singularities}, and a finite set of isolated points, also known as \textsl{Lefschetz singularities} (see Definition \ref{D:BLF}).  In this paper we determine the formal Poisson cohomology of a Poisson structure associated to a bLf. 

As shown in \cite{GSV}, on a bLf there is an associated Poisson structure $\pi$ whose degeneracy locus coincides with the singularity set of the fibration. Due to the existence of bLfs on 4--manifolds, such a Poisson structure $\pi$ exists on any homotopy class of maps from a 4-manifold $M$ to $S^2$, thus making their Poisson cohomology an interesting feature in terms of different classification questions. 

Around each type of singularity on a bLf, we use a different method for the computation of Poisson cohomology. This is due to the different model of the Poisson structure around each type (see equations  (\ref{eq:pi-gamma-with-k}), (\ref{eq:pi-C})). We explain our approach and connections to the works of Pelap \cite{P09} and Pichereau \cite{P06} in the next two paragraphs.

For the Poisson cohomology around Lefschetz singularities we use most of the 4-dimensional calculus for homology developed in Pelap \cite{P09}. This can be used here with amendments related to the Poisson structure around Lefschetz points. In particular, we show that to build compact formulas for the coboundary operator one needs a specific Clifford rotation $D$ of $\mathbb{R}^4$ that fixes the singularity (Remark \ref{fixD}), and an endomorphism $K$ of $\mathfrak{so}(4)$ that has a natural relation to $D$ (Remark \ref{fixK}) and is directly related to Jacobian Poisson structures (Remark \ref{Knabla}). These operators do not appear in \cite{P09} since we first obtain formulas for the Poisson coboundary operator for any Jacobian Poisson structure in 4 dimensions in Proposition \ref{compact}, while in 
\cite{P09} the focus was on an example with isolated singularity. These compact formulas are considerably different than those for homology in \cite{P09}, and cannot be extracted without $D$ and $K$. Having identified $D$ and $K$, then we restrict to the Poisson cohomology of Lefschetz singularities (Proposition \ref{compact-constant}). Similarly to  Pelap's work, our model has two weight homogeneous Casimirs forming a complete intersection with isolated singularity. Due to unimodularity of Jacobian Poisson structures, if follows that the cohomology spaces we want to compute here, should have the same rank as the Poisson homology spaces of the Sklyanin algebra. Computing the generators explicitly  is the more difficult part, and this is done in Propositions \ref{H0} - \ref{H2}.

For the Poisson cohomology around folds, our computation shows that in terms of cohomology, the singular circles can be viewed essentially as point singularities of a 3-manifold. In particular, the proof of Proposition \ref{cohomology circles} shows that one can decompose the coboundary operator in a way that isolates a point singularity in a 3-manifold and then adds an extra dimension in both the manifold and the singularity to get the fold singularity (circle) of a bLf. Restricting the Poisson structure on this 3-manifold, we get a Poisson structure determined by a weight homogeneous polynomial with isolated singularity as in Pichereau\cite{P06}. The weight homogenous polynomial that we use is naturally one of the Casimirs of the original Poisson structure. As a corollary, one could say that the computation of Poisson cohomology for fold singularities, can be thought of as an example of transferring the results in \cite{P06} to dimension 4. 
 The main result of the paper is the following.

\begin{theorem}\label{thm-main}
Let $f\colon M \rightarrow S^2$ be a broken Lefschetz fibration on an oriented, smooth, closed 4-manifold $M$. Denote by $\pi \in \mathfrak{X}^2(M)$ the associated Poisson structure vanishing on a finite collection of disjoint circles $\Gamma = \lbrace \gamma_1, \dots , \gamma_m\rbrace$ and a finite collection of points $C= \lbrace p_1, \dots , p_r \rbrace$.  The formal Poisson cohomology of $(M,\pi)$ on the tubular neighbourhood $U_\Gamma$ is determined by the following free $\mathrm{Cas}$- modules
\vspace{-0.3cm}

\begin{equation*}
\begin{split}
H^{0} (U_{\Gamma}, \pi) &\cong
\mathbb{R}^m
\\
H^{1} (U_{\Gamma}, \pi) &\cong \mathbb{R}^m\cong\bigoplus_{i=1}^m \frac{\partial}{\partial \theta_{i} }
\\
H^{2} (U_{\Gamma}, \pi) &\cong  0
\\
H^{3} (U_{\Gamma}, \pi) &\cong \mathbb{R}^m\cong \bigoplus_{i=1}^m \frac{\partial}{\partial x_{1_i}} \wedge  \frac{\partial}{\partial x_{2_i}} \wedge \frac{\partial}{\partial x_{3_i}}
\\
H^{4} (U_{\Gamma}, \pi) & \cong\mathbb{R}^m\cong  \bigoplus_{i=1}^m  \mathrm{vol}_i \quad ,
\end{split}
\end{equation*}
\vspace{-0.5cm}

\noindent where
\vspace{-0.1cm}

\begin{itemize}

\item $\mathrm{Cas}$ denotes the algebra $\mathbb{R}[Q^1_i, Q^2_i]$ of Casimirs with $Q^1_i=\theta_i, Q^2_i=-x_{1_i}^2+x_{2_i}^2+x_{3_i}^2$ the representatives around each $\gamma_i$, 
\item $\theta_i$ is the parameter of each circle $\gamma_i$ in $\Gamma$ with normal coordinates $(x_{1_i}, x_{2_i}, x_{3_i})$, 
\item $\mathrm{vol}_i$ is the volume form $d\theta\wedge dx_{1_i}\wedge dx_{2_i}\wedge dx_{3_i}$ around each $\gamma_i$.
\end{itemize}
Around $C$ the formal Poisson cohomology is determined by the following free $\mathrm{Cas}$- modules

\begin{equation*}
\begin{split}
H^{0} (U_{C}, \pi) &
\cong
\mathbb{R}^r
\\
H^{1} (U_{C}, \pi) &\cong
\mathbb{R}^r\cong  \bigoplus_{l=1}^r  E_l
\\
H^{2} (U_{C}, \pi) &\cong
\mathbb{R}^{6r}\cong \bigoplus_{l=1}^r\left( \left[\bigoplus_{k=1}^5K^{-1}(\nabla\nu_k\times\nabla P^1_l)\right] \oplus K^{-1}(\nabla P^1_l\times\nabla P^2_l) \right)
\\
H^{3} (U_{C}, \pi) &\cong\mathbb{R}^{13r}\cong \bigoplus_{l=1}^r\left( \left[ \bigoplus_{k=1}^{5} D \left( \nabla \nu_k \right)\right]  \oplus \left[ \bigoplus_{k=0}^{5} \nu_k D \left(\nabla P^2_l \right)\right] \oplus D(\nabla P^1_l)   
\oplus x_{1_l}x_{2_l}  \frac{}{} D(\nabla P^1_l)\right)
\\
H^{4} (U_{C}, \pi) & \cong\mathbb{R}^{7r}\cong \bigoplus_{l=1}^r\mathrm{span} \langle 1, \nu_1, \nu_2, \nu_3, \nu_4, \nu_5, \nu_6 \rangle  \quad , 
\end{split}
\end{equation*}

\noindent where 
\begin{itemize}
\item $\mathrm{Cas}$ denotes the algebra of Casimirs with $P^1_l=x_{1_l}^2-x_{2_l}^2+x_{3_l}^2-x_{4_l}^2, P^2_l=2(x_{1_l}x_{2_l}+x_{3_l}x_{4_l})$ the representatives around each point in $C$, 

\item $E=\sum_{i=1}^4x_i\partial_i$ is the Euler vector field in coordinates $(x_{1_l},x_{2_l},x_{3_l},x_{4_l})$, around each $p_l$,

\item $(\nu_k)_{0\leq k\leq 6}=(1, x_{1_l}, x_{2_l}, x_{3_l}, x_{4_l}, x_{1_l}x_{2_l}, x_{3_l} x_{4_l})$.

\end{itemize}
\end{theorem}

The Poisson structure on a bLf, together with the Poisson structure on near-symplectic 4-manifolds studied in \cite{BV} and log-symplectic structures on 4-manifolds, are examples of singular Poisson structures for any possible combination of degeneracies in the rank of a Poisson structure. The Poisson bivector associated to a bLF is of rank 2 or 0, whereas on a 4-manifold log-symplectic structures are Poisson structures of rank 4 or 2, and near-symplectic manifolds have a Poisson bivector of rank 4 or 0.  Hence, together with the Poisson cohomology of near-symplectic manifolds \cite{BV} and the Poisson cohomology of log-symplectic manifolds \cite{GMP,MO14i}, here we complete the Poisson cohomology computation for these structures.

The paper is structured as follows. In Section \ref{pre} we recall some basic notions of Poisson geometry and cohomology and in Section \ref{ops} we describe the operators that we will use in the computation of Poisson cohomology around Lefschetz singularities.  The latter is presented in section \ref{points}. Section \ref{sec:Poisson-BLF} provides a brief overview of Poisson structures on broken Lefschetz fibrations.  The paper continues with the general formulas for the Poisson coboundary operators of Jacobian Poisson structures in section \ref{Jacobian compact}. Finally, section \ref{section circles} concludes the paper with the computation of the formal Poisson cohomology around fold singularities.

\noindent {\bf Acknowledgements} 
We warmly thank Fani Petalidou and Anne Pichereau for comments on drafts of this work. We also thank the anonymous referees for valuable observations. R.V. acknowledges partial support by the FWO under EOS project G0H4518N.

\section{Preliminaries}
\subsection{Basic definitions and unimodular Poisson structures}\label{pre}
We first recall some basic objects from Poisson geometry, see \cite{LGPV13} for details.  A Poisson structure on  a smooth manifold $M$ is a Lie bracket $\{\cdot,\cdot\}$ on $C^\infty(M)$ satisfying the Leibniz rule $\{fg,h\}=f\{g,h\}+g\{f,h\}$.  Equivalently, a Poisson bivector field $\pi \in \mathfrak{X}^2(M)=\Gamma(\wedge^2 TM)$ is a bivector field satisfying $\left[ \pi, \pi \right]_{\mathrm{SN}}=0$ for the Schouten-Nijenhuis bracket $\left[ \cdot ,  \cdot \right]_{\mathrm{SN}} \colon \mathfrak{X}^{k}(M)  \times \mathfrak{X}^{l}(M) \rightarrow \mathfrak{X}^{k+l-1}(M)$. The Poisson bracket and bivector field are mutually determined by $\lbrace f, g \rbrace = \langle \pi, df \wedge dg \rangle$. We will use alternatively both expressions throughout the text. 

\noindent Let us now fix the notation and sign conventions for the Schouten-Nijenhuis bracket that we will use.  Fix a system of local coordinates on $M$ and consider $\zeta_i=\partial_{x_i}$ as an odd variable, so that $\zeta_i\zeta_j=-\zeta_j\zeta_i$. A $p$- vector field $P\in\mathfrak{X}^p(M)$ is then written as $\displaystyle P=\sum_{i_1<\cdots<i_p}P_{i_1\cdots i_p}\zeta_{i_1}\cdots\zeta_{i_p}$, with $P_{i_1\cdots i_p}\in C^\infty(M)$. Then for $Q\in\mathfrak{X}^q(M)$, define
\begin{equation}\label{SN}
[P,Q]_{SN}=\sum_i\partial_{\zeta_i}(P)\partial_{x_i}(Q)-(-1)^{(p-1)(q-1)}\partial_{\zeta_i}(Q)\partial_{x_i}(P).
\end{equation}

 \noindent where $\partial_{\zeta_{i_k}}\zeta_{i_1}\cdots\zeta_{i_p}=(-1)^{p-k}\zeta_{i_1}\cdots\widehat{\zeta_{i_k}}\cdots\zeta_{i_p}$. A bivector field $\pi$ induces an operator $\di_{\pi} \colon \mathfrak{X}^{\bullet}(M) \rightarrow \mathfrak{X}^{\bullet + 1}(M)$ by $\di_{\pi}(X) = \left[\pi,X \right]_{\mathrm{SN}}$, and if $\pi$ is Poisson then $\di^2_{\pi} = 0$. The pair $(\mathfrak{X}(M), \di_\pi)$  is called the {\em Lichnerowicz-Poisson cochain complex}, and 
\begin{equation}\label{defcoh}
H^{k}(M,\pi) := \frac{\ker\left( \di_{\pi}\colon \mathfrak{X}^{k}(M) \rightarrow \mathfrak{X}^{k+1}(M) \right) }{ \mathrm{Im}\left( \di_{\pi}\colon \mathfrak{X}^{k-1}(M) \rightarrow \mathfrak{X}^{k}(M) \right)},\;\;\;k=0,\ldots,\dim M,
\end{equation}
are the {\em Poisson cohomology spaces} of $(M, \pi)$.   The zeroth cohomology space $H^0(M,\pi)$ contains the \textsl{Casimir functions}, that is $f\in C^\infty(M)$ such that $\{f,g\}=0,\;\forall g\in C^\infty(M)$. The first cohomology group $H^1(M,\pi)$ is the quotient space of \textsl{Poisson} modulo \textsl{Hamiltonian} vector fields, i.e space of  $X\in \mathfrak{X}^1(M)$ satisfying $\mathcal{L}_{X} \pi = 0$ modulo the subspace of $X\in \mathfrak{X}^1(M)$ such that $X=\di_\pi(f)$.  Futhermore, $H^2(M,\pi)$ is the quotient of infinitesimal deformations of $\pi$ modulo trivial deformations, and $H^3(M,\pi)$ contains the obstructions to formal deformations of $\pi$.

Contraction with $\pi$ defines a vector bundle homomorphism $ \pi^{\sharp} \colon \Omega^1(M) \rightarrow \mathfrak{X}^1(M) $, usually referred to as the anchor map. Pointwise 
$\pi^{\sharp}_{p} (\alpha_p) = \pi_p(\alpha_p,\cdot)$ and it can be extended to a $C^\infty(M)$- linear homomorphism
\begin{equation}\label{sharp}
\wedge^\bullet\pi^\sharp:  \Omega^\bullet(M)\longrightarrow\mathfrak{X}^\bullet(M),
\end{equation}
which we denote again by $\pi^\sharp$.  The  Hamiltonian vector field of $f\in C^{\infty}(M)$ is then $X_f = \pi^{\sharp} (df)$.

Consider an orientable Poisson manifold with positive volume form $\Omega$. The vector field $Y^{\Omega}\colon C^{\infty}(M) \rightarrow C^{\infty}(M)$ defined by 
$$\mathcal{L}_{X_{f}} \Omega = (Y^{\Omega} f) \Omega$$

is a Poisson vector field  known as the {\em modular vector field} with respect to $\Omega$. One can check directly that there is a canonically defined Poisson cohomology class $\left[  Y^{\Omega} \right] $ called the \textsl{modular class} of $(M,\pi)$. If $\left[  Y^{\Omega} \right]=0$ then $(M,\pi)$ is called \textsl{unimodular}.

Let $\star$ denote the family of $C^\infty(M)-$ linear operators 
\begin{equation}\label{star}
\star: \mathfrak{X}^k(M)\to\Omega^{n-k}(M),\;\;\;\star X=\iota_X\Omega.
\end{equation} When $(M,\pi)$ is unimodular, $\star$ induces an isomorphism between the $k$-th Poisson cohomology group $H^k_\pi(M)$ and the $(n-k)$-th Poisson homology group $H^\pi_{n-k}(M)$  \cite[Proposition 4.18]{LGPV13}. 

\subsection{Operators in 4D}\label{ops}
The operators discussed in this section will be used in Sections \ref{Jacobian compact} and \ref{points}.
\subsubsection{Identifications in $\mathbb{R}^4$.}
We henceforth restrict in dimension $n=4$. Set $\partial_{i\cdots k}=\partial_i\wedge\cdots\wedge\partial_k$ and $\mathcal{A}=C^{\infty}(\mathbb{R}^4)$. Let also $\mathfrak{X}^k=\mathfrak{X}^k(\mathbb{R}^4)$  (respectively $\Omega^k=\Omega^k(\mathbb{R}^4)$) denote the spaces of $k$-vector fields (respectively $k$- differential forms) with coefficients from $\mathcal{A}$.  Identify $k-$ vector fields with the ordered tuples of their coefficient functions with the isomorphisms $\iota_k$, where

\begin{align}
\iota_1:\mathfrak{X}^1\stackrel{\simeq}{\longrightarrow}  \mathcal{A}^4,\;&\;\; X=\sum_{i=1}^4f_i\partial_i\mapsto (f_1,f_2,f_3,f_4)^T, \label{X1}
\\
\iota_2:\mathfrak{X}^2\stackrel{\simeq}{\longrightarrow}  \mathcal{A}^6,\;&\;\;U=\sum_{i<j=1}^4f_{ij}\partial_{ij}\mapsto (f_{12},f_{13},f_{14},f_{23},f_{24},f_{34})^T, \label{X2}
\\
\iota_3:\mathfrak{X}^3\stackrel{\simeq}{\longrightarrow}  \mathcal{A}^4,\;&\;\; W=\sum_{i<j<l=1}^4f_{ijl}\partial_{ijl}\mapsto (f_{123}, f_{124}, f_{134}, f_{234})^T,\label{X3} 
\\
\iota_0:  \mathfrak{X}^0\stackrel{\simeq}{\longrightarrow}\mathcal{A},\;&\;\;
\iota_4:\mathfrak{X}^4\stackrel{\simeq}{\longrightarrow}\mathcal{A}, \label{X0X4}
\\
\iota_0(f)=f,\;&\;\;\iota_4(f\partial_{1234})=f. \nonumber
\end{align}

\subsubsection{The operator $D$} 
\noindent
Consider automorphisms 
\[\mathcal{D}:\mathcal{A}^4\to \mathcal{A}^4,\;\;(f_1,f_2,f_3,f_4)\mapsto (g_1,g_2,g_3,g_4).\]
Via (\ref{X1}) and (\ref{X3}), there is one to one correspondence of such automorphisms with $\mathcal{A}-$ automorphisms 
\[\mathcal{D}_1:\mathfrak{X}^1\to\mathfrak{X}^1,\;\;\sum_{i=1}^4f_i\partial_i\mapsto\sum_{i=1}^4g_i\partial_i\] and $\mathcal{A}-$  automorphisms $\mathcal{D}_3: \mathfrak{X}^3\to\mathfrak{X}^3$ with 
\[f_1\partial_{123}+f_2\partial_{124}+f_3\partial_{134}+f_4\partial_{234}
\mapsto g_1\partial_{123}+g_2\partial_{124}+g_3\partial_{134}+g_4\partial_{234}
.\]

\noindent Let 
\begin{equation}\label{id}
I:\mathfrak{X}^k\to\Omega^k,
\end{equation} be the $\mathcal{A}-$ vector space isomorphism extending $\partial_i\mapsto\mathrm{d}x_i$, and $\mathcal{I}:\mathfrak{X}^1\to\mathfrak{X}^3$ be the $\mathcal{A}-$ isomorphism induced by $\mathcal{D}$, i.e.
\[\mathcal{I}\bigg(\sum_{i=1}^4f_i\partial_i\bigg)=g_1\partial_{123}+g_2\partial_{124}+g_3\partial_{134}+g_4\partial_{234}.\]
Then  $\mathcal{D}_1$ can be equivalently determined by the equation 
\[\mathcal{D}_1=I^{-1}\circ\star\circ\mathcal{I},\] 
and  $\mathcal{D}_3$ by the equation
\[\mathcal{I}=\mathcal{D}_3\circ\star^{-1}\circ I.\]
Fix $D_1$ and $D_3$  to be the corresponding $\mathcal{A}-$ automorphisms determined by the choice $\mathcal{D}=\mathrm{Id}_{\mathcal{A}^4}$. 
Since the context will be clear, we will denote them both as $D$. In matrix form,
\begin{equation}\label{D}
   D=
  \left[ {\begin{array}{cccc}
0 & 0 & 0 & 1 \\
0 & 0 & -1 & 0 \\
0 & 1 & 0 & 0 \\
-1 & 0 & 0 & 0 \\
  \end{array} } \right]
.\end{equation}

\noindent Throughout the text, if $X\in\mathfrak{X}^1$, we write $D(X)$ for the vector field $\iota_1^{-1}(D(\iota_1(X)))$ and also for the $4\times 1-$ matrix $D(\iota_1(X))$. In matrix form,

\begin{equation}
 X=\sum_{i=1}^4f_i\partial_i\stackrel{\iota_1}{\longrightarrow}\left[ {\begin{array}{c}
f_1 \\
f_2 \\
f_3 \\
f_4\\
\end{array} } \right]\mapsto D(X)=
\left[{\begin{array}{c}
f_4 \\
-f_3 \\
f_2 \\
-f_1\\
\end{array}}\right]
\stackrel{(\iota_1)^{-1}}{\longrightarrow} f_4\partial_1-f_3\partial_2+f_2\partial_3-f_1\partial_4.
\end{equation}

Similarly for some $Z\in\mathfrak{X}^3$; depending on the context, $D(Z)$ will stand for the 3-vector field $\iota_3^{-1}(D(\iota_3(Z)))$ or the $4\times 1-$ matrix $D(\iota_3(Z))$. In matrix form, for $Z=f_{123}\partial_{123}+f_{124}\partial_{124}+f_{134}\partial_{134}+f_{234}\partial_{234}$, it is

\begin{equation}
Z\stackrel{\iota_3}{\longrightarrow}\left[ {\begin{array}{c}
f_{123} \\
f_{124} \\
f_{134} \\
f_{234}\\
\end{array} } \right]\mapsto D(Z)=
\left[{\begin{array}{c}
f_{234} \\
-f_{134} \\
f_{124} \\
-f_{123}\\
\end{array}}\right]
\stackrel{(\iota_3)^{-1}}{\longrightarrow} f_{234}\partial_{123}-f_{134}\partial_{124}+f_{124}\partial_{134}-f_{123}\partial_{234}.
\end{equation}

\begin{remark}\label{fixD}
One can check that as an element of $SO(4)$, $D$ represents a right-isoclinic rotation of the 4-dimensional space. Since it has purely imaginary adjoint eigenvalues of multiplicity 2, it is a Clifford rotation. Such rotations do not have a fixed plane, however they do have a fixed point. In our use of $D$ in section \ref{points}, this point is identified with the Lefschetz singularity (Definition \ref{D:BLF}).
\end{remark}
\subsubsection{The operator $K$}  Equip $\mathbb{C}^4$ with the standard metric. For $k=2$, the associated Hodge operator $\ast: \Omega^k\to \Omega^{4-k}$ is the $\mathcal{A}-$ linear involution represented by the matrix

\begin{equation}\label{ast}
  \left[ {\begin{array}{cccccc}
0 & 0 & -1 & 0 & 0 & 0 \\
0 & 0 & 0 & 1 & 0 & 0 \\
-1 & 0 & 0 & 0 & 0 & 0 \\
0 & 1 & 0 & 0 & 0 & 0 \\
 0 & 0 & 0 & 0 & 0 & 1 \\
  0 & 0 & 0 & 0 & 1 & 0 \\
  \end{array} } \right].
\end{equation}
Consider now $\mathcal{A}$- automorphisms $\mathcal{K}: \mathcal{A}^6\to \mathcal{A}^6$  
\[\mathcal{K}:\;(f_{12},f_{13},f_{14},f_{23},f_{24}, f_{34})^T\mapsto(g_{12},g_{13},g_{14},g_{23},g_{24}, g_{34})^T.\] 
Via (\ref{X2}) there is a one to one correspondence with automorphisms (keep the same symbol) $\mathcal{K}:\mathfrak{X}^2\to\mathfrak{X}^2$ where \[\mathcal{K}:\;\sum_{1=i<j=4}f_{ij}\partial_{ij}\mapsto \sum_{1=i<j=4}g_{ij}\partial_{ij}.\]

Fix $K\;:\mathfrak{X}^2\to\mathfrak{X}^2$ to be the $\mathcal{A}$-automorphism satisfying the equation
\begin{equation}\label{k}
I^{-1}(\ast( I(K(U))))=K( I^{-1}(\star(U))).
\end{equation}

\noindent Using (\ref{star}),  (\ref{id}), (\ref{ast}), one has that this operator is represented by the $6\times 6$-matrix
\begin{equation}\label{K}
   K=
  \left[ {\begin{array}{cccccc}
0 & 0 & 0 & 1 & 0 & 0 \\
0 & 0 & 0 & 0 & 0 & 1 \\
0 & 0 & -1 & 0 & 0 & 0 \\
1 & 0 & 0 & 0 & 0 & 0 \\
 0 & 0 & 0 & 0 & 1 & 0 \\
  0 & -1 & 0 & 0 & 0 & 0 \\
  \end{array} } \right].
\end{equation}

\noindent Throughout the text, if $U\in\mathfrak{X}^2$, then $K(U)$ will stand for both the bivector field $\iota^{-1}_2(K(\iota_2(U)))$ and the $6\times 1-$ matrix $K(\iota_2(U))$ depending on the context. In matrix form, for $U=\sum_{i<j=1}^4 f_{ij}\partial_{ij}$,

\begin{equation}
U\stackrel{\iota_2}{\longrightarrow}\left[ {\begin{array}{c}
f_{12} \\
f_{13} \\
f_{14} \\
f_{23}\\
f_{24}\\
f_{34}\\
\end{array} } \right]\stackrel{K}{\longrightarrow}
\left[{\begin{array}{c}
f_{23} \\
f_{34} \\
-f_{14} \\
f_{12}\\
f_{24}\\
-f_{13}\\
\end{array}}\right]
\stackrel{(\iota_2)^{-1}}{\longrightarrow}f_{23}\partial_{12}+
f_{34}\partial_{13}-f_{14}\partial_{14}+f_{12}\partial_{23}+f_{24}\partial_{24}-f_{13}\partial_{34}.
\end{equation}

\begin{remark}\label{fixK}
Consider the complex vector space isomorphism between $\mathbb{C}^6$ and $\mathfrak{so}(4,\mathbb{C})$. For both vector spaces there are splittings in 6-dimensional real vector spaces $\mathbb{C}^6=\mathrm{Re}(\mathbb{C}^6)\oplus_\mathbb{R}\mathrm{Im}(\mathbb{C}^6)$ and $\mathfrak{so}(4,\mathbb{C})=\mathfrak{so}^-(4,\mathbb{C})\oplus_\mathbb{R}\mathfrak{so}^+(4,\mathbb{C})$, where the latter refers to anti-self dual and self dual operators. Then $D$ can be thought of as an element of $\mathfrak{so}^-(4,\mathbb{C})$, respectively $iD\in\mathfrak{so}^+(4,\mathbb{C})$.\newline
\noindent Consider the operator $K$ as an element of $SO(6)$. Given the splittings above, $K$ induces an automorphism $\mathfrak{K}$ of $\mathfrak{so}(4,\mathbb{C})$ that leaves $\mathfrak{so}^{\pm}(4,\mathbb{C})$ invariant. In fact, $\mathfrak{K}(D)=-D=\ast D$ and $\mathfrak{K}(iD)=iD=\ast(iD)$. To further motivate $K$ (and thus $D$)  in the setting of Jacobian Poisson structures and their cohomology, we refer to Remark \ref{Knabla}.
\end{remark}

In the sequence we will change between realizations of $k-$ vector fields and $k\times 1-$ matrices without other notice in most cases.

\subsubsection{Other operations in $\mathbb{R}^4$}
  In sections \ref{Jacobian compact} and \ref{points} we will use the involution 
\[\phi:\mathfrak{X}^2\to\mathfrak{X}^2,\;\;\;\;\phi=I^{-1}\circ\ast\circ I\]
which we consider equivalently from  (\ref{X2}) as an automorphism of $\mathcal{A}^6$. Its matrix coincides with the one of $\ast: \Omega^2\to\Omega^2$, see (\ref{ast}).

Given the identification (\ref{X2}), the wedge product $\mathfrak{X}^p\times\mathfrak{X}^q\to\mathfrak{X}^{p+q}$ is given for $p=q=1$ by the matrix formula

\begin{equation}\label{xsBV}
\bigg(X=\left[ {\begin{array}{c}
x_1 \\
x_2 \\
x_3 \\
x_4\\
  \end{array} } \right],
\;\; Y=\left[ {\begin{array}{c}
y_1 \\
y_2 \\
y_3 \\
y_4\\
  \end{array} } \right]\bigg)\mapsto 
X \times_{\iota_2} Y=\left[ {\begin{array}{c}
x_1y_2-x_2y_1 \\
x_1y_3-x_3y_1 \\
x_1y_4-x_4y_1 \\
x_2y_3-x_3y_2 \\
x_2y_4-x_4y_2 \\
x_3y_4-x_4y_3 \\
  \end{array} } \right].
\end{equation}

\begin{remark}\label{PvsBV1}
In \cite[Section 1.4]{P09}, the identification of 2-forms with 6-tuples of coefficient functions is done in another way than the one suggested by (\ref{X2}), namely through the map
\begin{equation} 
i_2:\;\sum_{i<j=1}^4f_{ij}\di _{ij}\mapsto \left[ {\begin{array}{c}
f_{14} \\
f_{12} \\
-f_{23}\\
f_{34} \\
-f_{13} \\
f_{23} \\
  \end{array} } \right],
\end{equation}
denoting $\di_{ij}:=\di x_i\wedge\di x_j$. The wedge product of two 1-forms $\alpha, \beta$ in $\mathbb{R}^4$ is then given by 

\begin{equation}\label{xsforms}
\bigg(\alpha=\left[ {\begin{array}{c}
\alpha_1 \\
\alpha_2 \\
\alpha_3 \\
\alpha_4\\
  \end{array} } \right],
\;\; \beta=\left[ {\begin{array}{c}
\beta_1 \\
\beta_2 \\
\beta_3 \\
\beta_4\\
  \end{array} } \right]\bigg)\mapsto 
\alpha\times_{i_2} \beta=\left[ {\begin{array}{c}
\alpha_1\beta_4-\alpha_4\beta_1 \\
\alpha_1\beta_2-\alpha_2\beta_1 \\
\alpha_3\beta_2-\alpha_2\beta_3 \\
\alpha_3\beta_4-\alpha_4\beta_3 \\
\alpha_3\beta_1-\alpha_1\beta_3 \\
\alpha_2\beta_4-\alpha_4\beta_2 \\
  \end{array} } \right].
\end{equation}
To compare the operations $\times_{\iota_2}$ and $\times_{i_2}$  in terms of bivectors $U\in\mathfrak{X}^2$, note that
\[i_2(I^{-1}(U))=\phi(K(\iota_2(U))),\]
and so for $X,Y\in\mathfrak{X}^1$,
\begin{equation}\label{newop}
I^{-1}(X)\times_{i_2} I^{-1}(Y)=\phi(K(X\times_{\iota_2}Y)).
\end{equation}
To help the interested reader compare the formulas in sections \ref{Jacobian compact}, \ref{points} with the formulas in \cite{P09}, we will keep the operation $\times_{i_2}$ between two vector fields. Alternatively, one can translate all formulas in those sections, using the ordinary wedge product (\ref{xsBV}) via (\ref{newop}). So from now on we will use the following formula for the wedge product between two vector fields,

\begin{equation}\label{xs}
\bigg(X=\left[ {\begin{array}{c}
x_1 \\
x_2 \\
x_3 \\
x_4\\
  \end{array} } \right],
\;\; Y=\left[ {\begin{array}{c}
y_1 \\
y_2 \\
y_3 \\
y_4\\
  \end{array} } \right]\bigg)\mapsto 
X\times Y=\left[ {\begin{array}{c}
x_1y_4-x_4y_1 \\
x_1y_2-x_2y_1 \\
x_3y_2-x_2y_3 \\
x_3y_4-x_4y_3 \\
x_3y_1-x_1y_3 \\
x_2y_4-x_4y_2 \\
  \end{array} } \right],
\end{equation}
where $x_i,y_j\in\mathcal{A}$.
\end{remark}

Denote by $\bar{\times}$ the operator 
\[\bar{\times} : \mathcal{A}^4\times  \mathcal{A}^6\to  \mathcal{A}^4,\;\;(X,U)\mapsto D\big(X\wedge K^{-1}(\phi^{-1}(U))\big).\]
In matrix form, using identifications $X=[x_1,x_2,x_3,x_4]^T, U=[f_{12},f_{13},f_{14},f_{23},f_{24},f_{34}]^T$, with $x_i,f_{ij}\in\mathcal{A}$, it is 

\begin{equation}\label{xbar}
   X\bar{\times} U=
  \left[ {\begin{array}{c}
-x_4f_{14}+x_2f_{23}-x_3f_{34} \\
x_3f_{12}-x_1f_{23}+x_4f_{24}\\
 -x_2f_{12}+x_4f_{13}+x_1f_{34}\\
 -x_3f_{13}+x_1f_{14}-x_2f_{24}\\
  \end{array} } \right].
\end{equation}
\begin{remark}\label{PvsBV2}
Similarly to Remark \ref{PvsBV1}, note that in \cite{P09} the identification of 3-forms to quadruples of coefficient functions is done in another way than the one suggested by (\ref{X3}), namely through the map

\begin{equation}
i_3:\;\sum_{i<j<k=1}^4f_{ijk}\di_{ijk} \mapsto \left[ {\begin{array}{c}
f_{234} \\
-f_{134} \\
f_{124}\\
-f_{123} \\
\end{array} } \right].
\end{equation}
When writing everything in terms of vectors $Z=\sum_{i<j<k=1}^4f_{ijk}\partial_{ijk}$, it is $i_3(I^{-1}(Z))=D(\iota_3(Z))$. We consider the operation (\ref{xbar}) to make again the relation with the corresponding operation in \cite{P09} (keeping in mind the different choice of ordering of coefficients).

\end{remark}

Denoting the gradient of functions by  
\[\displaystyle\nabla: \mathfrak{X}^0\to \mathfrak{X}^1,\;\;f\mapsto \sum_{i=1}^4\partial_i(f)\partial_i,\]
one may also define 
 \begin{equation}\label{nablaxbar}
   \nabla\times X=
  \left[ {\begin{array}{c}
\partial_1(x_4)-\partial_4(x_1) \\
\partial_1(x_2)-\partial_2(x_1) \\
\partial_3(x_2)-\partial_2(x_3) \\
\partial_3(x_4)-\partial_4(x_3)\\
\partial_3(x_1)-\partial_1(x_3)\\
\partial_2(x_4)-\partial_4(x_2)\\
  \end{array} } \right],\;\;\;
\nabla\bar{\times} U= 
  \left[ {\begin{array}{c}
-\partial_4(f_{14})+\partial_2(f_{23})-\partial_3(f_{34}) \\
\partial_3(f_{12})-\partial_1(f_{23})+\partial_4(f_{24})\\
 -\partial_2(f_{12})+\partial_4(f_{13})+\partial_1(f_{34})\\
 -\partial_3(f_{13})+\partial_1(f_{14})-\partial_2(f_{24})\\
  \end{array} } \right],
\end{equation}
  
\noindent following the sense of (\ref{xs}), (\ref{xbar}). \newline
\noindent In terms of notation, we will briefly (in Proposition \ref{compact} and Lemma \ref{2}) make use also of the vector dot product in a specific formula: Let $A,B$ be $4\times 1-$ matrices.  Then $\nabla(A)\cdot B$ is the $4\times 1-$ matrix
\begin{equation}\label{vectordot}
\textbf{grad}(A)^TB=
  \left[ {\begin{array}{cccc}
\partial_1(A_1) &  & \cdots & \partial_4(A_1)\\
\vdots &  & \cdots & \vdots\\
\vdots &  & \cdots & \vdots\\
 \partial_1(A_4) &  & \cdots & \partial_4(A_4)\\
  \end{array} } \right]^T  
   \left[ {\begin{array}{c}
B_1 \\
 \\
\vdots \\
 \\
B_4\\
  \end{array} } \right].
\end{equation}
Finally, denote the usual divergence as 
\[\mathrm{Div}:  \mathfrak{X}^1\to \mathcal{A},\;\;\;\;X=[f_1,f_2,f_3,f_4]^T\mapsto \sum_{i=1}^4\partial_i(f_i).\]

\begin{proposition}\cite{P09}\label{prop:calc-properties}
The operators defined above have the following properties.

\begin{align}
& \phi(u) \cdot y = u \cdot \phi(y),  &\textnormal{for } u, y \in   \mathcal{A}^6\label{Prop-Calc-2}
\\
&\phi(u) \cdot \phi(y) = u\cdot y,  &\textnormal{for } u, y \in   \mathcal{A}^6 \label{Prop-Calc-3}
\\
&u \cdot (y \times z) = y \cdot( z \, \bar{\times} \,  \phi(u)),  &\textnormal{for } u\in   \mathcal{A}^6, \, y,z \in   \mathcal{A}^4\label{Prop-Calc-4}
\\
& (u\times z) \cdot \phi(u\times y) = 0,   &\textnormal{for } u,y,z \in   \mathcal{A}^4 \label{Prop-Calc-8}
\\
& u \, \bar{\times} \, (y\times z) = y \, \bar{\times} \, (z \times u) , 
 &\textnormal{for } u,y,z \in   \mathcal{A}^4 \label{Prop-Calc-9}
\\
& z\, \bar{\times} \, \phi(u\times y) = -(z\cdot u) y + (z\cdot y)u, \quad &\textnormal{for } u,y,z \in   \mathcal{A}^4 \label{Prop-Calc-11}
\\
& (u \, \bar{\times} \, z) \, \bar{\times} \, \phi(u\times y) = - (z\cdot \phi(u\times y)) u , &\textnormal{for }  u,y \in   \mathcal{A}^4, z \in   \mathcal{A}^6  \label{Prop-Calc-14}
\\
&\nabla \, \bar{\times} \, (u\times y) = y \, \bar{\times} \, (\nabla \times u) -u \, \bar{\times} \, (\nabla \times y), &\textnormal{for }  u,y \in   \mathcal{A}^4 \label{Prop-Calc-15}
\\
&\nabla \times Fu = \nabla F \times u + F (\nabla \times u) , F\in  V, &\textnormal{for } u \in   \mathcal{A}^4 \label{Prop-Calc-16}
\\
&\nabla \, \bar{\times} \, Fy =  \nabla F \, \bar{\times} \, y + F (\nabla  \, \bar{\times} \, y)  , &\textnormal{for } F \in  \mathcal{A}, y \in   \mathcal{A}^6 \label{Prop-Calc-17}
\\ 
&\mathrm{Div}(Fu) = \nabla F \cdot u + F\mathrm{Div}(u), &\textnormal{for } F \in  \mathcal{A}, u \in   \mathcal{A}^4 \label{Prop-Calc-18}
\\
&\mathrm{Div}(u\, \bar{\times} \, y) = y \cdot \phi(\nabla \times u) - u \cdot (\nabla \, \bar{\times} \, y) , &\textnormal{for } u \in   \mathcal{A}^4, y \in   \mathcal{A}^6. \label{Prop-Calc-19}
\end{align}

\end{proposition}
\begin{proof}
Direct computation.
\end{proof}

\section{Poisson structure on broken Lefschetz fibrations}\label{sec:Poisson-BLF}

\begin{definition}
\label{D:BLF}
On a smooth, closed 4-manifold $M$, a {\it broken Lefschetz fibration} or {\it bLf} is a smooth map $f\colon M \rightarrow S^2$ that is a submersion outside a singularity set $C\sqcup \Gamma$. The allowed singularities are of the following type:
\begin{enumerate}
\item {\it Lefschetz} singularities:  finitely many points  $$C= \{ p_1, \dots , p_r\} \subset M,$$  which are locally modeled by complex charts
$$ \mathbb{C}^{2} \rightarrow \mathbb{C}  ,  \quad \quad (z_1, z_2) \mapsto z_{1}^{2} + z_{2}^{2},$$
\item {\it indefinite fold} singularities, also called {\it broken}, contained in the smooth embedded 1-dimensional submanifold $\Gamma \subset M \setminus  C$, and which are locally modelled by the real charts
$$ \mathbb{R}^{4} \rightarrow \mathbb{R}^{2} ,  \quad \quad  (\theta,x_1,x_2,x_3) \mapsto (\theta, - x_{1}^{2} + x_{2}^{2} + x_{3}^{2}).$$
\end{enumerate}
\end{definition}

In \cite{GSV} it was shown that a singular Poisson structure $\pi$ of rank 2 can be associated to the fibration structure of a bLf in such a way that the fibres of $f$ correspond to the leaves of the foliation induced by $\pi$ and the singularity set of $f$ is precisely the singular locus of the bivector. We recall the statement. 

\begin{theorem}\cite{GSV}\label{thm-GSV}
Let $M$ be a closed oriented smooth 4-manifold. On each homotopy class of maps from $M$ to the 2-sphere there exists a complete Poisson structure of rank 2 on $M$ whose associated Poisson bivector vanishes only on a finite collection of circles and isolated points.  
\end{theorem}

The local model of $\pi$ around the singular locus $\Gamma$ is given by

\begin{equation} \label{eq:pi-gamma-with-k}
\pi_{\Gamma_h}= h(\theta, x_1, x_2, x_3)  \left( x_1 \frac{\partial}{\partial x_2} \wedge \frac{\partial}{\partial x_3} + x_2 \frac{\partial}{\partial x_1} \wedge \frac{\partial}{\partial x_3} - x_3 \frac{\partial}{\partial x_1} \wedge \frac{\partial}{\partial x_2} \right).
\end{equation}
where $h$ is a non-vanishing function. Around the points of $C$ the local model is given by

\begin{align}\label{eq:pi-C}
\begin{split}
\pi_{C_h}=h(x_1, x_2, x_3, x_4)
&\left[(x_3^2+x_4^2)\frac{\partial}{\partial x_1}\wedge \frac{\partial}{\partial x_2}+ (x_2x_3-x_1x_4)\frac{\partial}{\partial x_1}\wedge \frac{\partial}{\partial x_3}\right.
\\ 
&- \left.(x_1x_3+x_2x_4) \frac{\partial}{\partial x_1}\wedge \frac{\partial}{\partial x_4}
+(x_1x_3+x_2x_4)\frac{\partial}{\partial x_2}\wedge \frac{\partial}{\partial x_3}\right.
\\ 
&+ \left.(x_2x_3-x_1x_4)\frac{\partial}{\partial x_2}\wedge \frac{\partial}{\partial x_4}+(x_1^2+x_2^2)\frac{\partial}{\partial x_3}\wedge \frac{\partial}{\partial x_4}\right],
\end{split}
\end{align}
where $h$ is again a non-vanishing function.

%
%

\section{Coboundary formulas for Jacobian Poisson structures in $4$ dimensions.}\label{Jacobian compact}

A Poisson structure on $\mathbb{R}[x_1,\cdots,x_n]$ is called Jacobian (\cite{Dam89}, attributed to Flaschka and Ratiu)  if there are $n-2$ generic polynomial functions $P_1,\cdots, P_{n-2}$ such that the Poisson bracket of two coordinate functions is given by
\begin{equation}\label{Jac}
\{x_i,x_j\}_\mu=\mu(x_1,\cdots,x_n)\frac{\di x_i\wedge\di x_j\wedge \di P_1\wedge\cdots\wedge\di P_{n-2}}{\di x_1\wedge\cdots\wedge\di x_n}.
\end{equation}
Denote by $\pi_\mu$ the bivector field corresponding to $\{\cdot,\cdot\}_\mu$. Obviously the $P_i$'s are Casimirs of $\pi_\mu$. It is easily checked that Jacobian structures are examples of unimodular Poisson structures and so the family (\ref{star}) of isomorphisms $\star$  induces a family of  isomorphisms
\[H^k(\mathbb{R}^n, \pi_\mu)\stackrel{\simeq}{\rightarrow} H_{n-k}(\mathbb{R}^n, \pi_\mu),\]
between Poisson cohomology and Poisson homology.

The Hamiltonian vector fields of the coordinate functions for (\ref{Jac}) are
 \begin{equation}\label{hamm}
 X_i^\mu=<\pi_\mu,\di x_i>,\;\;i=1,\ldots,n.
 \end{equation}  From (\ref{hamm}) one has that
 \begin{equation}\label{formulaham}
 X_i^\mu=\sum_{j=1}^n\{x_i,x_j\}_{\mu}\partial_j=\mu\sum_{j=1}^n\{x_i,x_j\}_{1}\partial_j,\;\;i=1,\ldots,n,
 \end{equation}
where $\{\cdot,\cdot\}_{1}:=\{\cdot,\cdot\}_{\mu=1}$.

The next Proposition contains compact general formulas for the coboundary operators in the Poisson cohomology of such structures  in dimension $n=4$. The proof splits in Lemmata \ref{0} - \ref{3}. We henceforth drop the subscript $\pi$ from the operators $\di^k_\pi:\mathfrak{X}^k(\mathbb{R}^4)\to\mathfrak{X}^{k+1}(\mathbb{R}^4)$ in (\ref{defcoh}).

\begin{proposition}\label{compact}
 The coboundary operators for the Poisson cohomology with smooth coefficients, of Jacobian Poisson structures in $\mathbb{R}^4$, are

\begin{equation}\label{d0-operator}
\di^0(g)=  \mu\nabla g\;\bar{\times}(\nabla P_1\times\nabla P_2)
\end{equation}
\begin{equation}\label{d1-operator}
\di^1(Y)=K^{-1}\left[ \phi\left(\di^0\times Y\right)  -   Y(\mu \nabla P_1\times\nabla P_2)\right]
\end{equation}
\begin{equation}\label{d2-operator}
\di^2(W) = D \left[ - \di^0\bar{\times}\phi(K(W))  +     \frac{1}{\mu}\di^0(\mu)\bar\times \phi(K(W))   +     \nabla(\mu\nabla P_1\times\nabla P_2)\cdot \phi(K(W))\right]
\end{equation}
\begin{equation}\label{d3-operator}
\di^3(Z) = - \mu\left( \nabla \times D(Z)\right) \cdot \phi\left(\nabla P_1 \times \nabla P_2 \right)  -  \frac{1}{\mu} D(Z)\cdot\di^0(\mu).
\end{equation}
\end{proposition}

On the left side of the formulas (\ref{d0-operator}) to  (\ref{d3-operator}), the notation refers to vector fields, while given the definitions in section \ref{ops}, on the right side there are matrices of coefficient functions.  One should use the isomorphisms $\iota_k$ (\ref{X1}) to  (\ref{X0X4})  in section \ref{ops} to pass from one side to the other. We chose to suppress them in the Proposition's statement to ease the notation, however we will indicate their use in the statements and proofs of the following Lemmata.

\begin{lemma}\label{0}
Given the assumptions of Proposition \ref{compact}, let $g\in \mathcal{A}$ and $\iota_1\colon \mathfrak{X}^1\to\mathcal{A}^4$. Then
\[\iota_1(\di^0(g))=   \mu\nabla g\;\bar{\times}(\nabla P_1\times\nabla P_2).\]
\end{lemma}

\begin{proof} By (\ref{formulaham}) we can equivalently write $\displaystyle \pi_\mu=   \frac{1}{2}\sum_{i=1}^n\partial_i\wedge X_i^\mu$. One then has
\begin{align}\label{d0 form}
\di^0(g)=&[\pi_\mu,g]_{SN}\stackrel{(\ref{SN})}{=}\sum_{i=1}^4\partial_{\zeta_i}(\pi_\mu)\partial_i(g)+0
\\
\nonumber
=&   -     \frac{1}{2}\sum_{i=1}^4\sum_{s=1}^4\{x_i,x_s\}_\mu\partial_s\wedge\partial_i(g)    +  \frac{1}{2}\sum_{i=1}^4\sum_{k=1}^4\{x_k,x_i\}_\mu\partial_k\wedge \partial_i(g)
\\
\nonumber
=& 
 -    \sum_{i,k=1}^4\partial_i(g)\{x_i,x_k\}_\mu\partial_k\stackrel{(\ref{formulaham})}{=}
\sum_{k=1}^4X_k^\mu(g)\partial_k.
\\
\nonumber
\end{align}

By (\ref{Jac}) it is
\begin{equation}\label{P1P2}
   \nabla P_1\times\nabla P_2=\mu^{-1}
  \left[ {\begin{array}{c}
\{x_2,x_3\}_\mu \\
\{x_3,x_4\}_\mu  \\
-\{x_1,x_4\}_\mu  \\
\{x_1,x_2\}_\mu  \\
\{x_2,x_4\}_\mu  \\
-\{x_1,x_3\}_\mu  \\
  \end{array} } \right]
  =
  \left[ {\begin{array}{c}
\{x_2,x_3\}_{1} \\
\{x_3,x_4\}_{1}  \\
-\{x_1,x_4\}_{1}  \\
\{x_1,x_2\}_{1}  \\
\{x_2,x_4\}_{1}  \\
-\{x_1,x_3\}_{1}  \\
  \end{array} } \right].
\end{equation}
Then using (\ref{xbar}), (\ref{P1P2}) one gets that 
\[\nabla g\bar{\times}(\nabla P_1\times\nabla P_2)=\displaystyle \iota_1\bigg(\sum_{i=k}^4\mu^{-1} X_k^\mu(g)\partial_k\bigg).\] The claim follows from (\ref{d0 form}).
\end{proof} 

\noindent We will often use the equation $\partial_{\zeta_i}(\pi_\mu)=-X^\mu_i$ in the next Lemmata.
\begin{lemma}\label{1}
Given the assumptions of Proposition \ref{compact}, let $Y\in\mathfrak{X}^1$ and $\iota_2\colon\mathfrak{X}^2\to\mathcal{A}^6$. Then
\[\iota_2(\di^1(Y))=K^{-1}\big[ \phi\left(\di^0\times Y\right)  -  Y\big(\mu \nabla P_1\times\nabla P_2)\big)\big].\]
\end{lemma}

\begin{proof} Let $Y=\sum_{i=1}^4f_i\partial_i$. Then $\di^1(Y)=[\pi_\mu,Y]_{SN}= A-B$ where from (\ref{SN}),
\[A=\sum_{i=1}^4\partial_{\zeta_i}(\pi_\mu)\partial_{x_i}(Y)   =       \sum_{i,j=1}^4X_i^\mu(f_j)\partial_{ij},\]
and
 \[B=\sum_{i=1}^4\partial_{\zeta_i}(Y)\partial_{x_i}(\pi_\mu)  = \sum_{i<j=1}^4\left[\sum_{k=1}^4f_k\partial_k(\{x_i,x_j\}_\mu)\right]\partial_{ij}.\]
The formulas above mean that for example the coefficient of $\partial_{23}$ contributed by $A$ is $(X_2^\mu(f_3)   -    X_3^\mu(f_2))\partial_{23}$, while the contribution of $B$ is $ \sum_{k=1}^4f_k\partial_k(\{x_2,x_3\}_\mu)\partial_{23}$.

Use (\ref{d0 form}) to identify the operator $\di^0$ with the vector (of vector fields)
\begin{equation}\label{d0vector}
\di^0=[X_1^\mu,     X_2^\mu,     X_3^\mu,     X_4^\mu]^T\in(\mathcal{A}^4)^4,
\end{equation}
meaning that each entry is a vector field and so is determined by four functions in $\mathcal{A}$ \footnote{In other words, one here might think of $\di^0$ as the square matrix determining $\pi_\mu$, i.e $\di^0\in\mathcal{A}^{4\times 4}$.}. Then by (\ref{ast}), (\ref{K}), (\ref{xs}) and (\ref{d0-operator}), one computes that 

 \[\iota_2(A)=     K^{-1}\big(\phi(\di^0\times Y)\big),\]
 
\noindent where $\di^0\times Y$ is understood in the sense of (\ref{nablaxbar}) replacing $\partial_i$ with $X_i^\mu$.

For the term $B$, let $Y$ act as a linear differential operator on each entry of the $6\times 1$ matrix $\mu(\nabla P_1\times\nabla P_2)$ given by (\ref{P1P2}). For example the first entry of the $6\times 1$ matrix $Y\big(\mu(\nabla P_1\times\nabla P_2)\big)$ is the function $\displaystyle \sum_{i=1}^4f_i\partial_i(\{x_2,x_3\}_\mu)$. Then $K^{-1}$ sends this matrix precisely to the part of $\di^1(Y)$ contributed by $B$, i.e.
 \[\iota_2(B)=   K^{-1}\left[Y\big(\mu(\nabla P_1\times\nabla P_2)\big)\right]. \]\end{proof}

\begin{lemma}\label{2}
Given the assumptions of Proposition \ref{compact}, let $W\in\mathfrak{X}^2$ and $\iota_3\colon\mathfrak{X}^3\to\mathcal{A}^4$. Then
\[\iota_3(\di^2(W)) =   D \left[    -    \di^0\bar{\times}\phi(K(W))   +     \frac{1}{\mu}\di^0(\mu)\bar\times \phi(K(W))    +     \nabla\big(\mu\nabla P_1\times\nabla P_2)\big)\cdot \phi(K(W))\right].\]
\end{lemma}

\begin{proof}
Let $W=\sum_{i<j=1}^4f_{ij}\partial_{ij}$. Then $\di^2(W)=[\pi_\mu,W]_{SN}=A+B$ where after a short computation with (\ref{SN}) we get
\[A=\sum_{i=1}^4\partial_{\zeta_i}(\pi_\mu)\partial_{x_i}(W)=        \sum_{\textrm{cycl}(i,j,s)} X_i^\mu(f_{js})\partial_{ijs},\]
and
\begin{equation}\label{d2B}
B=\sum_{i=1}^4\partial_{\zeta_i}(W)\partial_{x_i}(\pi_\mu)=   -  \sum_{\textrm{cycl}(i,j,s)}\big[f_{ks}\partial_k(\{x_i,x_j\}_\mu)-f_{sk}\partial_k(\{x_i,x_j\}_\mu)\big]\partial_{ijs}.
\end{equation}
Here, the sum for $A$ is taken cyclically on $(i,j,s)$, for example the function contributed by $A$ to the coefficient of $\partial_{123}$ in $\di^2(W)$ is 
\[\big(  X_1^\mu(f_{23})   -   X_2^\mu(f_{13})   +   X_3^\mu(f_{12})\big)\partial_{123}.\]
The sum for $B$ is taken also cyclically on $(i,j,s)$. For example, to compute the function contributed by $B$ to the coefficient of $\partial_{123}$, compute the given expression in (\ref{d2B}) for $\partial_{ijs}=\partial_{123}, \partial_{312},\partial_{231}$ (see also (\ref{luck})). One gets that for $\partial_{312}$, it is $i=3,j=1,s=2$ and so $f_{ks}$ is $f_{12}$ and $f_{sk}$ can be $f_{23}$ and then $f_{24}$\footnote{In our notation of the coefficient functions, $f_{ks}$ implies that $k<s$.}. 

Considering $\di^0$ as a vector (of vector fields) by (\ref{d0vector}), a direct computation in the sense of (\ref{nablaxbar}) replacing $\partial_i$ with $-X^\mu_i$, shows that 
\[\iota_3(A)=-D[\di^0\bar{\times}\phi(K(W))].\]

On the other hand we will show that 
\begin{equation}\label{i2B}
\iota_3(B)=  D\big[ (\nabla\mu\bar{\times}(\nabla P_1\times\nabla P_2))\bar\times \phi(K(W))+\nabla(\mu\nabla P_1\times\nabla P_2)\cdot\phi(K(W))\big].
\end{equation}
\noindent The first term in the bracket of the right hand side is understood through the operations defined in section \ref{ops} and then one can rewrite this term using (\ref{d0-operator}). The second term is the only instance in the paper where we use (\ref{vectordot}).  

\noindent Since the entire computation is long, we just give brief details on how to recover a particular coefficient function in (\ref{i2B}); all others follow similarly. Thus we focus on the coefficient function of $\partial_{123}$ contributed by the term $B$. By (\ref{d2B}) and the discussion below, this function is 

\begin{align}\label{luck}
 & - \partial_1(\{x_1,x_2\}_\mu) f_{13}   -     \partial_2(\{x_1,x_2\}_\mu) f_{23}    +   \partial_4(\{x_1,x_2\}_{\mu}) f_{34}
\\
\nonumber
&   -    \partial_1(\{x_3,x_1\}_\mu)f_{12}    +     \partial_3(\{x_3,x_1\}_\mu)f_{23}    +    \partial_4(\{x_3,x_1\}_\mu)f_{24}
\\
\nonumber
& +\partial_2(\{x_2,x_3\}_\mu)f_{12}    +    \partial_3(\{x_2,x_3\}_\mu)f_{13}    +     \partial_4(\{x_2,x_3\}_\mu)f_{14},
\end{align}

 \noindent where the first line is written out of the term in $B$ corresponding to $\partial_{ips}=\partial_{123}$, the second line from the term $\partial_{ips}=\partial_{312}$ and the third from the term $\partial_{ips}=\partial_{231}$. Then look at the function multiplied with e.g. $f_{23}$. Computing directly with (\ref{Jac}), and since the Casimirs are polynomial and their second derivatives commute,
\begin{align}
& \big[\partial_2(\{x_1,x_2\}_\mu)-\partial_3(\{x_3,x_1\}_\mu)\big]f_{23}
\\
\nonumber
=& \big[\partial_2(\mu)(\nabla P_1\times\nabla P_2)_4-\partial_3(\mu)(\nabla P_1\times\nabla P_2)_6+\mu\partial_4\big((\nabla P_1\times\nabla P_2)_3\big)\big]f_{23},
\\
\nonumber
\end{align}
where $M_i$ stands for the $i-$th entry of the column matrix $M$. Computing similarly for the coefficients of $f_{12}$ and $f_{13}$, one concludes that the function (\ref{luck}) is precisely
\[    \big[\big(\nabla\mu\bar{\times}(\nabla P_1\times\nabla P_2)\big)\bar{\times}\phi(K(W))+\nabla(\mu\nabla P_1\times\nabla P_2)\cdot \phi(K(W))\big]_4.\]
Then applying $D$ adjusts matrix entries and signs.
   \end{proof}

\begin{lemma}\label{3}
Given the assumptions of Proposition \ref{compact}, for $Z\in\mathfrak{X}^3$, and $\iota_4\colon\mathfrak{X}^4\to\mathcal{A}$, it is
\[\iota_4(\di^3(Z))  =   -  \mu\left( \nabla \times D(Z)\right) \cdot \phi\left(\nabla P_1 \times \nabla P_2 \right)     -  \frac{1}{\mu} D(Z)\cdot\di^0(\mu).\]
\end{lemma}

\begin{proof} Let $Z=\sum_{i<j<s=1}^4f_{ijs}\partial_{ijs}$ and set $\bar{f}_k=f_{ijs}$, where $k$ is the number completing the quadraple $\{1,2,3,4\}$  once $i<j<s$ are fixed. Then $\di^3(Z)=[\pi_\mu,Z]_{SN}= (A-B)\partial_{1234}$, where 
\[A=\sum_{i=1}^4\partial_{\zeta_i}(\pi_\mu)\partial_{x_i}(Z)=     \sum_{i,k=1}^4(-1)^k<X_i^\mu,\di x_k>\partial_i(\bar{f}_k),\]
and
\[B=\sum_{i=1}^4\partial_{\zeta_i}(Z)\partial_{x_i}(\pi_\mu)=   \sum_{i,k=1}^4(-1)^{k}\partial_i\big(\{x_i,x_k\}_\mu\big)\bar{f}_k.\]
Writing again $\displaystyle X_i^\mu=\sum_{k=1}^4\{x_i,x_k\}_\mu\partial_k$,
the function $A$ is then
\[A=    \sum_{i,k=1}^4(-1)^k\{x_i,x_k\}_\mu\partial_i(\bar{f}_k)  =   -    \sum_{i,k=1}^4(-1)^k\{x_k,x_i\}_\mu\partial_i(\bar{f}_k)=
     -     \sum_{k=1}^4(-1)^kX_k^\mu(\bar{f}_k).\]

To show that $A=   -    \mu(\nabla\times D(Z))\cdot\phi(\nabla P_1\times\nabla P_2)$, one computes directly that given (\ref{xs}), (\ref{ast}) and (\ref{P1P2}), 
\[(\nabla\times D(Z))\cdot\phi(\nabla P_1\times\nabla P_2)=
\mu^{-1}\bigg[\sum_{k=1}^4(-1)^kX^\mu_k(\bar{f}_k)\bigg],\]
and so half of (\ref{d3-operator}) is now proved.

For the other half,  since the Casimirs are polynomials their second derivatives commute. Then by (\ref{Jac}), for each fixed $k=1,\ldots,4$, it is
\begin{equation}\label{jac}
\sum_{i=1}^4\partial_i(\{x_i,x_k\}_\mu)=\sum_{i=1}^4\text{sgn}(\epsilon_{ik})\partial_i(\mu)\bigg(\partial_s(P_1)\partial_j(P_2)-\partial_j(P_1)\partial_s(P_2)\bigg).
\end{equation}
 Here, $\text{sgn}(\epsilon_{ik})$ denotes the sign of the permutation $\epsilon_{ik}=(i,k,s,j)$ of $S_4$ for $s<j$.
A direct computation with (\ref{D}), (\ref{xbar}) and (\ref{P1P2}), shows that $B$ is precisely the function $D(Z)\cdot\left(\nabla\mu\bar{\times}(\nabla P_1\times\nabla P_2)\right)$. Then use (\ref{d0-operator}) to get the second half of (\ref{d3-operator}).
\end{proof}

\begin{remark}\label{Knabla}
It is clear that the bivector field $W$ with $\iota_2(W)=\nabla P_1\times\nabla P_2$ plays a significant role in these formulas. To relate $W$ with the operators introduced in section \ref{ops}, observe that by (\ref{K}), (\ref{formulaham}) and (\ref{P1P2}),  it is
\begin{equation}\label{Knablaformula}
\nabla P_1\times\nabla P_2=K(\pi_{\mu=1}).
\end{equation}
\end{remark}

\section{Poisson cohomology around Lefschetz singular points.}\label{points}
In this section we compute the Poisson cohomology groups of the Poisson structure around a Lefschetz singularity of a bLf. The zeroth, first, and fourth cohomology groups will be computed directly using the unimodularity of the Poisson structure and results from \cite{P09}. The second and third cohomology groups require some more work but the approach for both is similar. 

We first prove that they are finitely generated as modules over the algebra of Casimir functions by writing every element of the corresponding kernel as a finite sum. Then we apply a reduction procedure to compute the generators for each group.

In order to simplify the formulas of the coboundary operators in Proposition \ref{compact}, we will choose the function $h$ in the formula  (\ref{eq:pi-C}) of the model  $\pi_{C_h}$ to be constant and equal to $h=1$. In particular, the Casimirs of $\pi_{C_1}$ are given by the real and imaginary parts of the parametrization of the Lefschetz singularities in Definition \ref{D:BLF}, namely 
\begin{equation}\label{casms}
P_1=x_1^2-x_2^2+x_3^2-x_4^2\;\;,\;P_2=2(x_1x_2+x_3x_4).
\end{equation} 
A simple comparison shows that the function $\mu$ in (\ref{Jac}) is then constant $\displaystyle\mu=\frac{1}{4}$.

\begin{proposition}\label{compact-constant}
For $P_1, P_2$ as in (\ref{casms}), the coboundary operators of the Poisson cohomology of the model (\ref{eq:pi-C}) are given by the following formulas
\begin{equation}\label{d0-operatorc}
\di^0(g)=   \frac{1}{4}\nabla g\;\bar{\times}\big(\nabla P_1\times\nabla P_2\big)
\end{equation}

\begin{equation}\label{d1-operatorc}
\di^1(Y)=\frac{1}{4}K^{-1}\left[ \mathrm{Div}(Y)\nabla P_1\times \nabla P_2+\nabla\times\bigg(Y\bar{\times}\phi(\nabla P_1\times \nabla P_2)\bigg)\right]
\end{equation}

\begin{equation}\label{d2-operatorc}
\di^2(W) = \frac{1}{4} D \left[ \big( \nabla \bar{\times} K (W) \big) \bar{\times} \phi\left(\nabla P_1 \times \nabla P_2  \right)  \, + \frac{}{}    \nabla \bigg( K (W) \cdot \phi\left( \nabla P_1 \times \nabla P_2\right) \bigg)\right]
\end{equation}

\begin{equation}\label{d3-operatorc}
\di^3(Z) =  -   \frac{1}{4}\big( \nabla \times D(Z)\big) \cdot \phi\left(\nabla P_1 \times \nabla P_2 \right).
\end{equation}
Alternatively, $\displaystyle \di^3(Z)=  -  \frac{1}{4}\mathrm{Div}\bigg[D(Z)\bar{\times}(\nabla P_1\times\nabla P_2)\bigg]$.

\end{proposition}

 \begin{proof}
The formulas for $\di^0$ and $\di^3$ are immediate from Proposition \ref{compact} since now $\mu=\frac{1}{4}$. For $\di^1$ and $\di^2$ the formulas are not taken directly from Proposition \ref{compact}, however similar and straightforward computations as in Lemmata \ref{1} and \ref{2} confirm the claims. For the alternative formula of $\di^3$, start from (\ref{d3-operatorc}) and use (\ref{Prop-Calc-2}),  (\ref{Prop-Calc-15}) ,  (\ref{Prop-Calc-19}).
 \end{proof}
 
Henceforth we simplify the notation by setting $\displaystyle \pi_{C_1}=\pi_C$. \\
\noindent Fix $\overline{\omega}=(1,1,1,1)$ as a weight vector inducing the polynomial degree so that $\overline{\omega}(P_1)=\overline{\omega}(P_2)=2$, and set also $V=\mathbb{R}[x_1,x_2,x_3,x_4]$. With a direct check, one sees that $P_i$ is not a zero divisor in $V/\langle P_j\rangle$ for $i\neq j=1,2$. Hence $(P_1,P_2)$ are, by definition, a regular sequence in $V$. Setting $J$ to be the ideal generated by $P_1,P_2$ and the $2\times 2$ minors of their Jacobian matrix , we have that $V_\textrm{sing}=V/J$ is finite dimensional and so by definition $(P_1,P_2)$ form a complete intersection with isolated singularity (at zero) \cite[Section 2]{P09}.

\noindent The Poisson homology groups of $\pi_C$ have Poincar\'{e} series given in \cite[Theorem 3.1]{P09}, and  they will have the same rank as free $\mathbb{R}[P_1,P_2]$-modules with the homology groups therein. Due to unimodularity of Jacobian Poisson structures, the rank of the Poisson cohomology groups is thus determined using the family of isomorphisms $\star$ (\ref{star}).
 
\noindent The proofs of the next Propositions compute the generators of the Poisson cohomology groups $H^k(\pi_C, B^4)$ with polynomial coefficients on a neighbourhood $U_C\approx B^4$ of a Lefschetz singularity.  Since the Poisson coboundary operator $\di$ is homogeneous quadratic, one can replace $V$ by $V_\textrm{formal}=\mathbb{R}[[x_1, x_2, x_3,x_4]]$, the algebra of formal power series equipped with $\pi_C$ and thus get the formal Poisson cohomology i.e the second list of Theorem \ref{thm-main}.

\begin{proposition}\label{H0}
The formal Poisson cohomology group $H^0(U_C,\pi_C)$ is a rank 1 free $\mathbb{R}[P_1,P_2]-$ module generated by $1$.
\end{proposition}

\begin{proof}
By \cite[Theorem 3.2]{P09} and since the Poisson structure is unimodular, $H^0(U_C,\pi_C)$  is a rank 1 free $\mathbb{R}[P_1,P_2]-$ module generated by $\star^{-1}(\di x_1\wedge\di x_2\wedge\di x_3\wedge\di x_4)=1$.
\end{proof}

\begin{proposition}\label{H1}
The formal Poisson cohomology group $H^1(U_C,\pi_C)$ is a rank 1 free $\mathbb{R}[P_1,P_2]-$ module generated by the Euler vector field.
\end{proposition}
\begin{proof}
 Let 
 \[\rho=\sum_{i=1}^4(-1)^{i-1}x_i\di x_1\wedge\cdots\wedge\widehat{\di x_i}\wedge\cdots\wedge\di x_4,\] and $E=\sum_{i=1}^4x_i\partial_i$ be the Euler vector field. Then $\star E=\rho$. Since $\mathrm{Div}(E)=4$ is equal to $2\deg(P_1)$, we get the claim by \cite[Theorem 3.3]{P09}.
\end{proof}

\begin{proposition}\label{H4}
The formal Poisson cohomology group $H^4(U_C,\pi_C)$  is a free $\mathbb{R}[P_1,P_2]-$ module of rank 7, generated by $(\nu_i)_{0\leq i\leq 6}=(1, x_1, x_2, x_3, x_4, x_1x_2, x_3 x_4)$. 
\end{proposition}

\begin{proof}
Since $\di^3$ is a quadratic operator, no constant or linear multiples of $\partial_{1234}$ are in $\mathrm{Im}(\di^3)$. The fact that the kernel of $\di^4$ is all of $V$, together with a direct computation of the image $\di^3(Z)$, where $Z\in\mathfrak{X}^3(U_C)$ is a linear 3-vector field,  proves the claim.
\end{proof}

We now proceed to compute $H^3(U_C,\pi_C)$. The result appears in Proposition \ref{endH3}, which in turn is the combination of Propositions \ref{H3} and \ref{H3second}.

\begin{proposition}\label{H3}
The formal Poisson cohomology group $H^3(U_C,\pi_C)$  is a free $\mathbb{R}[P_1,P_2]-$ module contained in the $\mathbb{R}[P_1,P_2]-$ module 

\begin{equation}\nonumber
\bigoplus_{k=1}^{6}\mathbb{R}[P_1, P_2]  D \left( \nabla \nu_k \right)  +  \bigoplus_{k=0}^{6} \mathbb{R}[P_1, P_2]  \nu_k D \left(\nabla P_1 \right)  + \bigoplus_{k=0}^{6} \mathbb{R}[P_1, P_2]  \nu_k  D \left(\nabla P_2 \right).
\end{equation} 

\end{proposition}

\begin{proof}

Let $H \in \ker(\di^3)$. Then by \eqref{d3-operatorc}, 

$$\left(\nabla \times D(H) \right) \cdot \phi\left(\nabla P_1\times \nabla P_2\right) = 0.$$
 Due to the unimodularity of Jacobian Poisson structures we can use \cite[Prop 3.2]{P09} and so, 
\begin{equation}
D(H) = \beta_1 \nabla P_1 + \beta_2 \nabla P_2 + \nabla \beta_3 \quad , \quad\text{for some}\;\; \beta_i \in V .
\end{equation}
Since $D$ is an isomorphism and $D^2 = -\mathrm{Id}$, it is
\begin{equation}\label{what is H}
 H =- D \left( \beta_1 \nabla P_1 \right) - D \left( \beta_2 \nabla P_2 \right)-  D \left( \nabla \beta_3 \right) \quad , \quad \beta_i \in V. 
 \end{equation}
\noindent The idea of this proof is to compute $\mathrm{mod}\;\mathrm{Im}(\di^2)$, all the terms on the right hand side of (\ref{what is H}) and show that they are some finite sums. We start by computing the first two summands on the right hand side of (\ref{what is H}).  

By definition $H^4(U_C,\pi_C)=V/\mathrm{Im}(\di^3)$, so one can write the polynomials $\beta_l$ as a sum of an element in $\mathrm{Im}(\di^3)$ and a representative of a class in $H^4(U_C, \pi_C)$. By Proposition \ref{H4},  $H^4(U_C, \pi_C)$ is a Casimir-module its  generators being the elements $ 1, x_1, x_2, x_3, x_4, x_1x_2, x_3 x_4$. Ignoring the sign of \eqref{d3-operatorc} since $\di^3$ is $\mathbb{R}-$ linear, one can write that for $l=1,2$, it is

\begin{equation}\label{form}
 \beta_l =     \frac{1}{4} \left( \nabla \times D(H_l) \right) \cdot \phi(\nabla P_1 \times \nabla P_2 ) + \sum_{i=0}^{\mu_l} \sum_{j=0}^{\delta_l} \sum_{k=0}^{6}  \lambda_{ijk}^{l} P_1^{i} P_2^{j} \nu_k,
 \end{equation}

 \noindent for some  $H_l\in\mathfrak{X}^3$, and then

\begin{equation}\label{eq:a_l}
\beta_l \nabla P_l =   \frac{1}{4}  \left( \nabla \times D(H_l) \right) \cdot \phi(\nabla P_1 \times \nabla P_2 ) \nabla P_l + \sum_{i=0}^{\mu_l} \sum_{j=0}^{\delta_l} \sum_{k=0}^{6}  \lambda_{ijk}^{l} P_1^{i} P_2^{j} \nu_k \nabla P_l.
\end{equation}

\noindent By Lemma \ref{H3claim} (proved right after this proof),  the 3-vector fields
\begin{equation}\nonumber
B_l:= \frac{1}{4} \left( \nabla \times D(H_l) \right) \cdot \phi\left(\nabla P_1 \times \nabla P_2 \right) \nabla P_l,\;\;l=1,2
\end{equation}
 satisfy the condition 
 \begin{equation}\label{cond}
D(B_l)\in \mathrm{Im(d^2)}.
\end{equation}
Given (\ref{eq:a_l}) and  (\ref{cond}) one has
\begin{equation}\label{bl}
-D \left( \beta_l \nabla P_l  \right) =   \di^2(W_l)  -  \sum_{i=0}^{\mu_l} \sum_{j=0}^{\delta_l} \sum_{k=0}^{6}  \lambda_{i,j,k}^l P_{1}^{i}  P_{2}^{j}  D(\nu_k  \nabla  P_l),\;\;\;l=1,2,
\end{equation}
for $W_l$ as in the proof of Lemma \ref{H3claim}.

We now want to show a formula like (\ref{bl}) also for the third summand in (\ref{what is H}), $D(\nabla\beta_3)$. Taking gradients on both sides of  (\ref{form}), we have

\[\nabla \beta_3=      \frac{1}{4}\nabla\left[\left( \nabla \times D(H_3) \right) \cdot \phi(\nabla P_1 \times \nabla P_2 )\right] +\sum_{k=1}^{6}\left( \sum_{i=0}^{\mu_3} \sum_{j=0}^{\delta_3}  \lambda_{ijk}^{3} P_1^{i} P_2^{j}\right) \nabla \nu_k\]
\[+\left( \sum_{i=0}^{\mu_3} \sum_{j=0}^{\delta_3} \sum_{k=0}^{6}i \lambda_{ijk}^{3} P_1^{i-1} P_2^{j}\nu_k\right) \nabla P_1+\left( \sum_{i=0}^{\mu_3} \sum_{j=0}^{\delta_3} \sum_{k=0}^{6}j \lambda_{ijk}^{3} P_1^{i} P_2^{j-1}\nu_k\right) \nabla P_2\]

Setting $W_3:=- K^{-1}(\nabla\times D(H_3))$  for some $H_3\in\mathfrak{X}^3$, observe that $\nabla\bar{\times} K(W_3)=0$ and so

$$ \di^2(W_3)=- \frac{1}{4}D\big[\nabla\left[\left( \nabla \times D(H_3) \right) \cdot \phi(\nabla P_1 \times \nabla P_2 )\right]\big].$$

Thus, 

\begin{align}
-D(\nabla \beta_3)=\di^2(W_3)- & \sum_{k=1}^{6}  \left( \sum_{i=0}^{\mu_3}  \sum_{j=0}^{\delta_3}  \lambda_{ijk}^{3} P_{1}^{i} P_{2}^{j} \right) D(\nabla \nu_k)  
\nonumber
\\
-& \left( \sum_{i=0}^{\mu_3} \sum_{j=0}^{\delta_3} \sum_{k=0}^{6}i \lambda_{ijk}^{3} P_1^{i-1} P_2^{j}\nu_k\right) D(\nabla P_1)
\nonumber
\\
-& \left( \sum_{i=0}^{\mu_3} \sum_{j=0}^{\delta_3} \sum_{k=0}^{6}j \lambda_{ijk}^{3} P_1^{i} P_2^{j-1}\nu_k\right) D(\nabla P_2)
\label{b3}
\end{align}

Therefore, by (\ref{what is H}), (\ref{bl}) and (\ref{b3}), we get that $\ker(\di^3) = \mathrm{Im}(\di^2) + L_3$, where
\begin{equation}\label{L_3}
L_3 =\bigoplus_{k=1}^{6}\mathbb{R}[P_1, P_2]  D \left( \nabla \nu_k \right)  +  \bigoplus_{k=0}^{6} \mathbb{R}[P_1, P_2]  \nu_k D \left(\nabla P_1 \right)  + \bigoplus_{k=0}^{6} \mathbb{R}[P_1, P_2]  \nu_k  D \left(\nabla P_2 \right).
\end{equation} 
\end{proof}

We now prove the claim that was postponed from the proof of Proposition \ref{H3}.

\begin{lemma}\label{H3claim}
For $l=1,2$, the 3-vector field
\[B_l:=     \frac{1}{4} \left( \nabla \times D(H_l) \right) \cdot \phi\left(\nabla P_1 \times \nabla P_2 \right) \nabla P_l\] 
 satisfies (\ref{cond}).

\end{lemma}

\begin{proof}
By (\ref{Prop-Calc-14}) we get that
\begin{equation}
B_l =    -    \frac{1}{4} \left[\frac{}{}\nabla P_l \bar{\times} \left( \nabla\times  D (H_l) \right) \right] \bar{\times}  \phi\left( \nabla P_1 \times \nabla P_2 \right) 
\end{equation}
Using (\ref{Prop-Calc-15}), 

$$B_l=    -    \frac{1}{4}\left[\frac{}{}\nabla \bar{\times} \left( D(H_l) \times \nabla P_l \right)  + D (H_l) \bar{\times}  (\nabla \times \nabla P_l)\right]\bar{\times}\phi(\nabla P_1\times\nabla P_2)$$

\noindent and since $\nabla\times\nabla f=0,\;\;\forall f\in V$,

\begin{equation}\label{end}
B_l=    -    \frac{1}{4}\left[\frac{}{}\nabla \bar{\times} \left( D(H_l) \times \nabla P_l \right) \right]  \bar{\times} \phi\left( \nabla P_1 \times \nabla P_2 \right).
\end{equation}

\noindent Observe also that due to (\ref{Prop-Calc-8}) one has

\begin{equation}\label{zero term}
\left(\nabla P_l \times D (H_l) \right) \cdot \phi\left(\nabla P_1 \times \nabla P_2 \right)=0.
\end{equation}

\noindent Consider $W_l :=K^{-1} \left( D (H_l) \times \nabla P_l \right)$. By \eqref{d2-operatorc}, (\ref{end}) and (\ref{zero term}) one then has that $D(B_l)= - \di^2(W_l)$ and (\ref{cond}) is proved.
\end{proof}

In the next Proposition we do a reduction procedure to eliminate $\mathbb{R}[P_1,P_2]-$ linearly dependent terms contained in the right hand side of the module $L_3$ in  (\ref{L_3}).

\begin{proposition}\label{H3second}
The $\mathbb{R}[P_1,P_2]-$ module $L_3$ (\ref{L_3})  is isomorphic to the following direct sum
\begin{align}
\left[ \bigoplus_{k=1}^{5} \mathbb{R}[P_1, P_2] D \left( \nabla \nu_k \right)\right] &\oplus \left[ \bigoplus_{k=0}^{5} \mathbb{R}[P_1, P_2]  \nu_k D \left(\nabla P_2 \right)\right] \oplus
\nonumber
\\
\left(\mathbb{R}[P_1,P_2] \frac{}{} D(\nabla P_1) \right) & \oplus \left( \mathbb{R}[P_1,P_2]x_1x_2  \frac{}{} D(\nabla P_1) \right),
\end{align}

\noindent where  $(\nu_i)_{0\leq i\leq 6}=(1, x_1, x_2, x_3, x_4, x_1x_2, x_3 x_4)$. 
\end{proposition}

\begin{proof}

\noindent Take $G_1=\phi(\nabla x_1\times E)\in\mathfrak{X}^2$ where $E$ is the Euler vector field $E=\sum_{i=1}^4x_i\partial_i$. Then
\begin{align}
(\nabla\bar{\times} G_1)\bar{\times}\phi(\nabla P_1\times\nabla P_2) &\stackrel{\eqref{Prop-Calc-11}}{=}-\left[\underbrace{(\nabla\cdot\nabla x_1)}_{=0}\cdot E-(\underbrace{\nabla\cdot E}_{\mathrm{Div}(E)})\cdot \nabla x_1\right]\bar{\times}\phi(\nabla P_1\times\nabla P_2) 
\nonumber
\\
&=(\mathrm{Div}(E)\cdot\nabla x_1)\bar{\times}\phi(\nabla P_1\times\nabla P_2) 
\nonumber
\\
&\stackrel{\eqref{Prop-Calc-11}}{=}-4(\nabla x_1\cdot\nabla P_1)\nabla P_2+4(\nabla x_1\cdot\nabla P_2)\nabla P_1
\nonumber
\\
&=-8(x_1\nabla P_2- x_2\nabla P_1)
\end{align}

On the other hand,
\begin{align}
\nabla \left[\frac{}{}G_1\cdot \phi(\nabla P_1\times \nabla P_2)\right]&\stackrel{\eqref{Prop-Calc-3}}{=}\nabla\left[(\nabla x_1\times E)\cdot (\nabla P_1\times \nabla P_2)\right]
\nonumber
\\
&\stackrel{\eqref{Prop-Calc-4}}{=}\nabla\left[\frac{}{}\nabla P_1\cdot(\nabla P_2\bar{\times} \phi(\nabla x_1\times E))\right]
\nonumber
\\
&\stackrel{\eqref{Prop-Calc-11}}{=}\nabla\left[\frac{}{}\nabla P_1\cdot \left(-(\nabla P_2\cdot\nabla x_1)\cdot E+(\nabla P_2\cdot E)\cdot\nabla x_1\right)\right]
\nonumber
\\
&=\nabla\left[\frac{}{}\nabla P_1\left(-2 x_2E+2P_2\nabla x_1\right)\right]
=-4\nabla\left( x_2P_1-x_1P_2\right)
\end{align}

\noindent where we used that $\nabla P_i\cdot E=2P_i$. Thus for $\overline{W}_1:=K^{-1}(G_1)$ we have that 
\[\di^2(\overline{W}_1)=-D(x_1\nabla P_2)+D( x_2\nabla P_1)-D(P_1\nabla x_2)+D(P_2\nabla x_1).\]

In the same way, taking $\overline{W}_i=K^{-1}(G_i)$, $i=2,3,4$ with $G_i=\phi(\nabla x_i\times E)$, a direct computation shows that respectively
\[\di^2(\overline{W}_2)=D(x_1\nabla P_1)+D( x_2\nabla P_2)-D(P_1\nabla x_1)-D(P_2\nabla x_2),\]
\[\di^2(\overline{W}_3)=D( x_4\nabla P_1)-D( x_3\nabla P_2)-D(P_1\nabla x_4)+D(P_2\nabla x_3),\]
\[\di^2(\overline{W}_4)=D( x_3\nabla P_1)+D( x_4\nabla P_2)-D(P_1\nabla x_3)-D(P_2\nabla x_4)\]
and so $D(x_i\nabla P_1),\;\;i=1,2,3,4$ are written as linear combinations of  elements in 
$$ 
\bigoplus_{k=1}^{6} \mathbb{R}[P_1, P_2] D \left( \nabla \nu_k \right),\;\;  \bigoplus_{k=0}^{6} \mathbb{R}[P_1, P_2]  \nu_k D \left(\nabla P_2 \right),\;\;\text{and}\;\;\mathrm{Im(d^2)}.
$$

Since $P_2=2(x_1 x_2+ x_3 x_4)$, it is
\begin{align}
D(\nabla x_3x_4) &=\frac{1}{2}D(\nabla P_2)-D(\nabla x_1x_2), \;\;\;\textrm{and}
\nonumber
\\
x_3x_4D(\nabla P_i) &=\frac{1}{2}P_2D(\nabla P_i)-x_1x_2D(\nabla P_i),\;\;i=1,2.
\nonumber
\end{align}

\noindent Thus $\mathrm{Ker(\di^3)}=\mathrm{Im(\di^2)}
+L_3^{'}$, where

\[L_3^{'}= \bigoplus_{k=1}^{5} \mathbb{R}[P_1, P_2] D \left( \nabla \nu_k \right)+ \bigoplus_{k=0}^{5} \mathbb{R}[P_1, P_2]  \nu_k D \left(\nabla P_2 \right)+\mathbb{R}[P_1,P_2]D(\nabla P_1)+\mathbb{R}[P_1,P_2]x_1x_2D(\nabla P_1).\]
\end{proof}

\begin{proposition}\label{endH3}
The formal Poisson cohomology group $H^3(U_C,\pi_C)$ in Theorem \ref{thm-main}  is a free $\mathbb{R}[P_1,P_2]-$ module of finite rank isomorphic to
\begin{align}
\left[ \bigoplus_{k=1}^{5} \mathbb{R}[P_1, P_2] D \left( \nabla \nu_k \right)\right] &\oplus \left[ \bigoplus_{k=0}^{5} \mathbb{R}[P_1, P_2]  \nu_k D \left(\nabla P_2 \right)\right] \oplus
\nonumber
\\
\left(\mathbb{R}[P_1,P_2] \frac{}{} D(\nabla P_1) \right) & \oplus \left( \mathbb{R}[P_1,P_2]x_1x_2  \frac{}{} D(\nabla P_1) \right),
\end{align}

\noindent where  $(\nu_i)_{0\leq i\leq 6}=(1, x_1, x_2, x_3, x_4, x_1x_2, x_3 x_4)$. 
\end{proposition}
\begin{proof}
By Propositions \ref{H3} and \ref{H3second}.
\end{proof}

We now show the statement of Theorem \ref{thm-main} concerning the second Poisson cohomology group around a Lefschetz singularity. The statements of the next technical Lemma will be used in the proof of Proposition \ref{finiteH2}.

\begin{lemma}\label{sidecalc}

\begin{enumerate}
\item Consider the vector field $Y_0$ with $\iota_1(Y_0) = D(H_0) \bar{\times} \left( \nabla P_1\times\nabla P_2 \right)$, for some $H_0\in\mathfrak{X}^3$. Then
\[\iota_2(\di^1(Y_0))=\frac{1}{4}K^{-1}\big[\big(\big(\nabla \times D(H_0)\big)\cdot \phi(\nabla P_1\times P_2)\big)\nabla P_1\times \nabla P_2\big].\]
\item Consider the vector fields $Y_l$ with $\iota_1(Y_l)=-\nabla P_l\bar{\times}(\nabla\times D(H_l)),\;l=1,2$ for some $H_l\in\mathfrak{X}^3$. Then
\[\iota_2(\di^1(Y_l))=\frac{1}{4}K^{-1}\left[\nabla\big((\nabla\times D(H_l))\cdot \phi(\nabla P_1\times \nabla P_2)\big)\times\nabla P_l\right].\]
\end{enumerate}
\end{lemma}

\begin{proof}
For the first point, observe that 
\begin{align}
\nabla\times \left[Y_0\, \bar{\times} \frac{}{}\phi(\nabla P_1\times \nabla P_2)\right] \stackrel{\eqref{Prop-Calc-9}}{=}&\nabla\times \left[\left(\nabla P_1\bar{\times}(\nabla P_2\times D(H_0))\right)\bar{\times}\phi(\nabla P_1\times \nabla P_2)\right]
\nonumber
\\
\stackrel{\eqref{Prop-Calc-14}}{=}&\nabla\times \left[\left((\nabla P_2\times D(H_0))\cdot \phi(\nabla P_2\times \nabla P_1)\right)\nabla P_1\right]
\nonumber
\\
\stackrel{\eqref{Prop-Calc-8}}{=}&0.
\nonumber
\end{align}

Furthermore it is 
\begin{align}
&\mathrm{Div}\left(D(H_0) \, \bar{\times} \frac{}{} (\nabla P_1\times\nabla P_2)\right)\nabla P_1\times\nabla P_2
\nonumber
\\
\stackrel{\eqref{Prop-Calc-19}}{=}& \left[\left(\nabla P_1\times\nabla P_2\right)\cdot \phi(\nabla \times D(H_0)) - D(H_0)\cdot \frac{}{} \left( \underbrace{\nabla\bar{\times}(\nabla P_1\times\nabla P_2}_{=0})\right)\right]\nabla P_1\times \nabla P_2
\nonumber
\\
\stackrel{\eqref{Prop-Calc-2}}{=}&\left[\left(\nabla \times D(H_0)\right)\cdot \frac{}{} \phi(\nabla P_1\times P_2)\right]\nabla P_1\times \nabla P_2.
\nonumber
\end{align}
Thus we get the first claim by (\ref{d1-operatorc}).

\noindent For the second point, we have
\begin{align}
\nabla\times\big[Y_l\bar{\times}\phi(\nabla P_1\times\nabla P_2)\big]
\stackrel{\eqref{Prop-Calc-14}}{=}&\nabla\times\big[\big((\nabla\times D(H_l))\cdot \phi(\nabla P_1\times\nabla P_2)\big)\nabla P_l\big]
\nonumber
\\
\stackrel{\eqref{Prop-Calc-16}}{=}&\nabla\big[\big(\nabla\times D(H_l)\big)\cdot \phi(\nabla P_1\times P_2)\big]\times\nabla P_l.
\nonumber
\end{align}

Since $\nabla\times\nabla P_l=0$ and $\nabla\bar{\times}(\nabla\times D(H_l))=0$, one has
\[\mathrm{Div}\big(\nabla P_l \, \bar{\times} \frac{}{}(\nabla\times D(H_l))\big)\nabla P_1\times P_2\stackrel{\eqref{Prop-Calc-19}}{=}0,\]
and thus we get the second claim by (\ref{d1-operatorc}).
\end{proof} 

\begin{proposition}\label{finiteH2}
The formal Poisson cohomology group $H^2(U_C,\pi_C)$  is a free $\mathbb{R}[P_1,P_2]-$ module contained in the $\mathbb{R}[P_1,P_2]-$ module
\[\mathbb{R}[P_1,P_2]\bigg[\nu_kK^{-1}(\nabla P_1\times\nabla P_2)\oplus K^{-1}\big[\nabla(P_{1}  P_{2}  \nu_k)\times\nabla P_1\big]\oplus K^{-1}\big[\nabla(P_{1} P_{2}  \nu_k)\times\nabla P_2\big]\bigg].\]
\end{proposition}

\begin{proof} Let $G\in\mathrm{Ker}(\di^2)$. By  \cite[Prop. 3.4]{P09}, the unimodularity of Jacobian Poisson structures, and the fact that $D$ is an isomorphism, one then has
\[K(G)=\beta_0\nabla P_1\times\nabla P_2+\nabla \beta_1\times\nabla P_1+\nabla \beta_2\times\nabla P_2,\]
for some $\beta_l\in V$. Since $K$ is a  $V-$ linear isomorphism,
\begin{equation}\label{what is G}
G=\beta_0K^{-1}(\nabla P_1\times\nabla P_2)+K^{-1}(\nabla \beta_1\times\nabla P_1)+K^{-1}(\nabla \beta_2\times\nabla P_2).
 \end{equation} 
 
\noindent The idea of this proof is to compute $\mathrm{mod}\;\mathrm{Im}(\di^1)$, all the terms on the right hand side of (\ref{what is G}) and show that they are some finite sums.

\noindent We start by computing the first summand on the right hand side of (\ref{what is G}).  
 
\noindent In the same way as for (\ref{form}) let 
\[\beta_l=\frac{1}{4}\big(\nabla\times D(H_l)\big)\cdot \phi(\nabla P_1\times P_2)+\sum_{i=0}^{\mu_l} \sum_{j=0}^{\delta_l} \sum_{k=0}^{6}  \lambda_{i,j,k}^l P_{1}^{i}  P_{2}^{j}  \nu_k,\]

\noindent where $H_l\in\mathfrak{X}^3$. Then
\begin{align}
\beta_0K^{-1}(\nabla P_1\times\nabla P_2)=& \frac{1}{4}\big[\big(\nabla\times D(H_0)\big)\cdot \phi(\nabla P_1\times P_2)\big]K^{-1}(\nabla P_1\times\nabla P_2)
\nonumber
\\
& +\big[\sum_{i=0}^{\mu_0} \sum_{j=0}^{\delta_0} \sum_{k=0}^{6}  \lambda_{i,j,k}^0 P_{1}^{i}  P_{2}^{j}  \nu_k\big]K^{-1}(\nabla P_1\times\nabla P_2),
\nonumber
\end{align}

\noindent and so by Lemma \ref{sidecalc} (1), 
\begin{equation}\label{beta_0}
\beta_0K^{-1}(\nabla P_1\times\nabla P_2)\equiv \big[\sum_{i=0}^{\mu_0} \sum_{j=0}^{\delta_0} \sum_{k=0}^{6}  \lambda_{i,j,k}^0 P_{1}^{i}  P_{2}^{j}  \nu_k\big]K^{-1}(\nabla P_1\times\nabla P_2)\;\;\;\;\; \mathrm{mod}\;\mathrm{Im}(\di^1).
\end{equation}

\noindent We calculate now the second and third summand of \eqref{what is G}. One has that for $l=1,2$, 

\[\nabla\beta_l=\frac{1}{4}\nabla\large[\big(\nabla\times D(H_l)\big)\cdot \phi(\nabla P_1\times P_2)\large]+\sum_{i=0}^{\mu_l} \sum_{j=0}^{\delta_l} \sum_{k=0}^{6}  \lambda_{i,j,k}^l \nabla(P_{1}^{i}  P_{2}^{j}  \nu_k).\]
By Lemma  \ref{sidecalc} (2), we get
\begin{equation}\label{betal}
K^{-1}(\nabla\beta_l\times\nabla P_l)\equiv K^{-1}\big[\sum_{i=0}^{\mu_l} \sum_{j=0}^{\delta_l} \sum_{k=0}^{6}  \lambda_{i,j,k}^l \nabla(P_{1}^{i}  P_{2}^{j}  \nu_k)\times\nabla P_l\big] \;\;\;\;\;\mathrm{mod}\;\mathrm{Im}(\di^1).
\end{equation} 

The claim then follows from (\ref{what is G}), (\ref{beta_0}) and (\ref{betal}). \end{proof}

We now want to compute explicitly the generators of $H^2(U_C,\pi_C)$. For this, we will do a reduction procedure to eliminate $\mathbb{R}[P_1,P_2]-$ linearly dependent terms contained in the right hand sides of (\ref{beta_0}) and (\ref{betal}) (and thus  (\ref{what is G})).

\begin{proposition}\label{H2}
The formal Poisson cohomology group $H^2(U_C,\pi_C)$ is a free $\mathbb{R}[P_1,P_2]-$ module isomorphic to
\[ \left[\bigoplus_{k=1}^5\mathbb{R}[P_1,P_2]K^{-1}(\nabla\nu_k\times\nabla P_1)\right]\oplus \mathbb{R}[P_1,P_2]K^{-1}(\nabla P_1\times\nabla P_2)\]
\noindent where  $(\nu_i)_{1\leq i\leq 6}=(x_1, x_2, x_3, x_4, x_1x_2, x_3 x_4)$. 
\end{proposition}

\begin{proof}

\noindent Consider the vector field    $Y:=\lambda_{ijk}^lP_1^iP_2^j\nu_kE\in\mathfrak{X}^1$. To compute $\di^1(Y)$ from \eqref{d1-operatorc},  one has
\begin{align}
\mathrm{Div}(Y)\nabla P_1\times\nabla P_2\stackrel{(\ref{Prop-Calc-18})}{=} & \lambda_{ikj}^l\big[(2i+2j+4)P_1^iP_2^j\nu_k+P_1^iP_2^j\nabla\nu_k\cdot E\big]\nabla P_1\times\nabla P_2\
\nonumber
\\
=& \big(2i+2j+4+\deg(\nu_k)\big)\lambda_{ijk}^lP_1^iP_2^j\nu_k\nabla P_1\times\nabla P_2. \label{mid1}
\end{align}
Furthermore,

\begin{align}
\nabla\times \big[Y\bar{\times}\phi(\nabla P_1\times\nabla P_2)\big] \stackrel{\eqref{Prop-Calc-11}}{=}&
\nabla\times\big[-\big(Y\cdot\nabla P_1\big)\nabla P_2+\big(Y\cdot\nabla P_2)\nabla P_1\big]
\nonumber
\\
=& \,
\nabla\times\big[-2\lambda_{ijk}^lP_1^{i+1}P_2^j\nu_k\nabla P_2+2\lambda_{ijk}^lP_1^iP_2^{j+1}\nu_k\nabla P_1\big]
\nonumber
\\
\stackrel{\eqref{Prop-Calc-16}}{=}&  2 \lambda_{ijk}^l\big[-\nabla(P_1^{i+1}P_2^j\nu_k)\times\nabla P_2+\nabla(P_1^iP_2^{j+1}\nu_k)\times\nabla P_1\big]
\nonumber
\\
=& 2 \lambda_{ijk}^l\big[-(i+1)P_1^iP_2^j\nu_k\nabla P_1\times\nabla P_2-P_1^{i+1}P_2^j\nabla\nu_k\times\nabla P_2
\nonumber
\\
& + (j+1)P_1^iP_2^j\nu_k\nabla P_2\times\nabla P_1+P_1^iP_2^{j+1}\nabla\nu_k\times\nabla P_1\big]
\nonumber
\\
=& - (2i+2j+4) \lambda_{ijk}^lP_1^iP_2^j\nu_k\nabla P_1\times\nabla P_2
\nonumber
\\
 & +   2 \lambda_{ijk}^l\big[P_1^iP_2^{j+1}\nabla\nu_k\times\nabla P_1-P_1^{i+1}P_2^j\nabla\nu_k\times\nabla P_2\big] \label{mid2}
\end{align}

\noindent Then, from (\ref{mid1}) and (\ref{mid2}), 
\begin{align}\label{morered}
\di^1(Y)=& \frac{1}{4}\lambda_{ijk}^l\deg(\nu_k)P_1^iP_2^j\nu_kK^{-1}(\nabla P_1\times\nabla P_2)
\nonumber
\\
& + \frac{1}{2}\lambda_{ijk}^l\big[P_1^iP_2^{j+1}K^{-1}(\nabla \nu_k\times\nabla P_1) - P_1^{i+1}P_2^jK^{-1}(\nabla\nu_k\times\nabla P_2)\big]
\end{align}

\noindent Equation (\ref{morered}) shows that in $H^2(U_C,\pi_C)$, the terms in the right hand side are linearly dependent. We eliminate the terms of the form $\mathbb{R}[P_1,P_2]\nu_kK^{-1}(\nabla P_1\times\nabla P_2)$ and so by (\ref{what is G}), (\ref{beta_0}),  (\ref{betal}) and (\ref{morered}), we get that $\mathrm{ker}(\di^2)=\mathrm{Im}(\di^1)+L_2$ where

\begin{equation}\label{l2}
L_2=\sum_{k=1}^6\mathbb{R}[P_1,P_2]K^{-1}(\nabla\nu_k\times\nabla P_1)+\sum_{k=1}^6\mathbb{R}[P_1,P_2]K^{-1}(\nabla\nu_k\times\nabla P_2)+\mathbb{R}[P_1,P_2]K^{-1}(\nabla P_1\times\nabla P_2).
\end{equation}

\noindent We eliminate more terms from $L_2$ as follows. Computing $\nabla\times\big[\nabla x_i\bar{\times} \phi(\nabla P_1\times\nabla P_2)\big]$ for $i=1,2,3,4$ we have respectively
\[\di^1(\nabla x_1)=\frac{1}{2}K^{-1}\big[\nabla x_2\times\nabla P_1-\nabla x_1\times\nabla P_2\big]\]
\[\di^1(\nabla x_2)=\frac{1}{2}K^{-1}\big[\nabla x_1\times\nabla P_1-\nabla x_2\times\nabla P_2\big]\]
\[\di^1(\nabla x_3)=\frac{1}{2}K^{-1}\big[\nabla x_4\times\nabla P_1-\nabla x_3\times\nabla P_2\big]\]
\[\di^1(\nabla x_4)=\frac{1}{2}K^{-1}\big[\nabla x_3\times\nabla P_1-\nabla x_4\times\nabla P_2\big].\]
Thus for $i=1,2,3,4$, the bivector $K^{-1}(\nabla x_i\times\nabla P_2)$ is written $\mathrm{mod\;Im}(\di^1)$ as linear combination of other elements in $L_2$ (\ref{l2}).

\noindent Since $\nabla x_{12}\times\nabla P_2=-\nabla x_{34}\times\nabla P_2$ and $\nabla P_1\times\nabla P_2=-\frac{1}{2}\nabla x_{12}\times\nabla P_1-\frac{1}{2}\nabla x_{34}\times\nabla P_1$ we get the claim. \end{proof}

%
%
\section{Poisson cohomology around singular circles}\label{section circles}
In this section we calculate the formal Poisson cohomology of the Poisson structure $\pi_{\Gamma_h}$ (\ref{eq:pi-gamma-with-k}) around the circles of fold singularities of a bLf. We restrict to the case where the function determining the conformal class is $h=1$ as in section \ref{points}, so we work with the linear model
\begin{equation}\label{eq:pi-gamma}
\pi_{\Gamma_1}=x_1 \frac{\partial}{\partial x_2} \wedge \frac{\partial}{\partial x_3} + x_2 \frac{\partial}{\partial x_1} \wedge \frac{\partial}{\partial x_3} - x_3 \frac{\partial}{\partial x_1} \wedge \frac{\partial}{\partial x_2},
\end{equation}
and drop the subscript $\Gamma$ henceforth.

On the normal bundle of a singular circle there is a splitting $\mathbb{R}^3 \times S^1 \rightarrow S^1$ into a rank 1-bundle and a rank 2-bundle over $S^1$. There are two possible splittings up to isotopy \cite{GSV}. One is orientable and the other non-orientable, where the former is given by the identity map and the latter is defined by the involution 
\begin{align}\label{involution}
\iota\colon S^1 \times D^3 &\rightarrow S^1 \times D^3
\\
(\theta, x_1, x_2, x_3) &\mapsto (\theta+ \pi, -x_1, x_2, -x_3).
\nonumber
\end{align}

\noindent The bivector field $\pi$ is invariant under $\iota$ and descends to the quotient of $S^1 \times B^3$ by the involution  for the non-orientable tubular neighbourhood \cite[Proposition 3.2]{GSV}.

\noindent For simplicity in the following calculations we rename $\theta=x_0$. Fix the volume form $\mathrm{vol}=\di x_0\wedge\di x_1\wedge \di x_2\wedge \di x_3$. A straightforward calculation shows that the modular vector field of $\pi$ vanishes identically.  Moreover, the extension of the Poisson structure to the regular parts is symplectic, thus the structure is unimodular everywhere and there exists a measure preserved by all Hamiltonian flows.

Let $\pi^\sharp: \Omega^1(\mathbb{R}^4)\longrightarrow \mathfrak{X}^1$ be again the contraction of $\pi$, i.e, $\pi^\sharp(\di x_i)(\di x_j)=\langle\pi,\di x_i\otimes\di x_j\rangle$. The Hamiltonian vector fields of the coordinate functions are
\begin{align}
\pi^\sharp(\di x_0)&=0,
\nonumber
\\
\pi^\sharp(\di x_1)&=x_2\partial_3-x_3\partial_2,
\nonumber
\\
\pi^\sharp(\di x_2)&=x_1\partial_3+x_3\partial_1,
\nonumber
\\
\pi^\sharp(\di x_3)&=-x_1\partial_2-x_2\partial_1.
\nonumber
\end{align}
For simplicity in notation we set $X_i:=\pi^\sharp(\di x_i)$.
\subsection{Description of the coboundary operator.} \label{equations d circles}
Let  $V=\mathbb{R}[x_0,x_1,x_2,x_3]$ be the algebra of polynomials in $x_0,x_1,x_2,x_3$. The Poisson bivector is
\[\pi=\frac{1}{2}\sum_{i=1}^3\partial_i\wedge X_i.\]
For $f\in V$,
\begin{equation}\label{d^0}
\di^0(f)=-\sum_{i=1}^3\partial_i(f)X_i=\sum_{i=1}^3 X_i(f)\partial_i.
\end{equation}
For $Y=\sum_{i=0}^3f_i\partial_i\in\mathfrak{X}^1$,

\begin{equation}\label{d^1}
\di^1(Y)= -\sum_{i=1}^3 X_i(f_0)\partial_{0i} +\sum_{i<j=1}^3
\left(X_i(f_j)-X_j(f_i)-(-1)^{[\frac{i+j}{2}]}f_k\right)\partial_{ij}
\end{equation}
where $[t]$ denotes the integral part of $t\in\mathbb{R}$, for example $[3.7]=[3]=3$ and the index $k$ is the index completing the triplet $\{i,j,k\}=\{1,2,3\}$ for chosen $i<j$.
Furthermore,  for $\displaystyle W=\sum_{i<j=0}^3f_{ij}\partial_{ij}\in\mathfrak{X}^2$,

\begin{equation}\label{d^2}
\di^2(W)=\sum_{i<j=1}^3\left (- X_i(f_{0j})+X_j(f_{0i})+(-1)^{[\frac{i+j}{2}]}f_{0k}\right)\partial_{0ij}
\end{equation}
\[-
\left(\sum_{i<j=1}^3(-1)^i X_i(f_{jk})\right)\partial_{123}\]
and finally, for $\displaystyle Z=\sum_{i<j<k=0}^3f_{ijk}
\partial_{ijk}\in\mathfrak{X}^3$,

\begin{equation}\label{d^3}
\di^3(Z)=\left[\sum_{i<j=1}^3(-1)^{k}X_k(f_{0ij})\right]\partial_{0123}.
\end{equation}

\subsection{Formal cohomology.}
 Let $V_i=\mathbb{R}_i[x_0,x_1,x_2,x_3]$ be the vector space of homogeneous polynomials of degree $i$ and $\mathfrak{X}^k_i$ be the space of $k$-vector fields whose coefficients are elements of $V_i$. Since $\pi$ is linear, when restricting to formal coefficients of $k-$ vector fields, one can decompose each term $\di^k:\mathfrak{X}^k\to\mathfrak{X}^{k+1}$ of section \ref{equations d circles} as $\di^k=\sum_{i=0}^\infty \di^k_i$ with $\di^k_i: \mathfrak{X}^{k}_i\rightarrow\mathfrak{X}^{k+1}_i$.

In terms of our notation for polyvector fields and their coefficient functions the operators $\di^\bullet_i$ fit in the sequence

\begin{equation}\label{seq}
0\longrightarrow V_i\stackrel{\di_i^0}{\longrightarrow} V_i^{\otimes 4}\stackrel{\di_i^1}{\longrightarrow} V_i^{\otimes 6}\stackrel{\di_i^2}{\longrightarrow} V_i^{\otimes 4}\stackrel{\di_i^3}{\longrightarrow} V_i\longrightarrow 0
\end{equation}
and more precisely

\begin{equation}\label{seqf}
f\stackrel{\di^0_i}{\longrightarrow} ( f_{0},f_{1},f_{2},f_{3}) \stackrel{\di^1_i}{\longrightarrow} ( f_{01},f_{02}, f_{03},f_{12},f_{13},f_{23}) \stackrel{\di^2_i}{\longrightarrow} ( f_{012},f_{013},f_{023},f_{123}) \stackrel{\di^3_i}{\longrightarrow} f_{0123}.
\end{equation}

One can check using the Jacobian form of $\pi$ that the functions
\begin{equation}\label{Q's}
Q_1(x_0,x_1,x_2,x_3)=x_0,\;\;\text{and}\;\;Q_2(x_0,x_1,x_2,x_3)=-x_1^2+x_2^2+x_3^2
\end{equation}
 parametrizing the singular locus, are the generators of the algebra of Casimir functions for $\pi$, which henceforth is denoted by $\mathbb{R}[Q_1,Q_2]$.
\begin{proposition}\label{cohomology circles}
Let $(U_{\Gamma}, \pi)$ the tubular neighbourhood of indefinite fold singularities of a bLf with Poisson bivector as in  \eqref{eq:pi-gamma}.  Let also $(\mathfrak{X}_{\textnormal{formal}}^{\bullet}(U_\Gamma), \di)$ be the Poisson cochain complex of multivector fields with formal coefficients.  The formal Poisson cohomology $H^{\bullet}_\mathrm{formal}(U_{\Gamma}, \pi)$ is given by the following list of free $\mathbb{R}[Q_1,Q_2]$-modules

\begin{itemize}
\item $H^0_\mathrm{formal}(U_{\Gamma}, \pi)=\mathbb{R}[Q_1,Q_2]$

\item  $H^1_\mathrm{formal}(U_{\Gamma}, \pi)=\mathbb{R}[Q_1,Q_2]\partial_0 $

\item  $H^2_\mathrm{formal}(U_{\Gamma}, \pi)=0$

\item  $H^3_\mathrm{formal}(U_{\Gamma}, \pi)=\mathbb{R}[Q_1,Q_2]\partial_{123} $

\item $H^4_\mathrm{formal}(U_{\Gamma}, \pi)=\mathbb{R}[Q_1,Q_2]\partial_{0123}$
\end{itemize}
\end{proposition}

\begin{proof} 

Since $\pi$ is linear, computing the cohomology $H^k_i(U_\Gamma, \pi)$ with coefficient functions of fixed polynomial degree $i$ will determine the formal Poisson cohomology  $H^k_\mathrm{formal}(U_{\Gamma}, \pi)$ by replacing $V_i$ with $V_{\textrm{formal}}=\mathbb{R}[[x_0,x_1,x_2,x_3]]$.

 We thus prove our claim for fixed polynomial degree $i$. Consider the maps

\[(d^1_1)^i: V_i\longrightarrow V_i^{\otimes 3},\;\;\;(d^2_1)^i: V_i^{\otimes 3}\longrightarrow V_i^{\otimes 3}\]

\noindent defined by 

\[f_0\stackrel{(d^1_1)^i}{\longrightarrow} ( f_{01},f_{02},f_{03}) \stackrel{(d^2_1)^i}{\longrightarrow} ( f_{012},f_{013}, f_{023}),\]

\noindent and the maps 
\[(d^1_2)^i: V_i^{\otimes 3}\longrightarrow V_i^{\otimes 3},\;\;\;(d^2_2)^i: V_i^{\otimes 3}\longrightarrow V_i\]

\noindent  defined by 

\[( f_1,f_2,f_3)\stackrel{(d^1_2)^i}{\longrightarrow} ( f_{12},f_{13},f_{23}) \stackrel{(d^2_2)^i}{\longrightarrow} f_{123}.
\]

\noindent For simplicity, we keep the same notation for the maps induced by the $(d_l^k)^i$ between vector fields, i.e. 

\[ (d_1^1)^i, \;(d_2^1)^i:\; \mathfrak{X}^1_i(\mathbb{R}^4)\to\mathfrak{X}^2_i(\mathbb{R}^4)\;\;\textrm{and}\;\;(d_1^2)^i, \;(d_2^2)^i:\; \mathfrak{X}^2_i(\mathbb{R}^4)\to\mathfrak{X}^3_i(\mathbb{R}^4).\]

One can split $\di^1_i$  and $\di^2_i$ as 
\begin{equation}\label{d_1}
\di^1_i=(d^1_1)^i+(d^1_2)^i
\end{equation}
\begin{equation}
\di^2_i=(d^2_1)^i+(d^2_2)^i.
\end{equation}

\noindent Observe that by equations \eqref{d^0}--\eqref{d^3}, the following diagrams commute

\begin{center}

\begin{align}\label{obs2}
\xymatrix{ 
	                   						&    \qquad  \qquad f_1 \partial_1 + f_2 \partial_2 + f_3 \partial_3\ \qquad  \ar[dr]^{\bullet \, \wedge \partial_0}   &\\
f_0  \ar[ur]^{d_i^0} \ar[dr]_{\bullet \, \wedge \partial_0}   	&   &  \quad  -(f_1 \partial_{01} + f_2 \partial_{02} + f_3 \partial_{03})\\
									&  f_0 \partial_0  \ar[ur]_{(d_1^1)^{i}}  &  
}
\end{align}
\end{center}

\begin{center}
\begin{align}\label{obs1}
\xymatrix{ 
	                   						&    f_{12} \partial_{12} + f_{13} \partial_{13} + f_{23} \partial_{23}  \ar[dr]^{\bullet \, \wedge \partial_0}   &
\\
f_1 \partial_1 + f_2 \partial_2 + f_3\partial_3  \ar[ur]^{(d^1_2)^i} \ar[dr]_{\bullet \, \wedge \partial_0}   	&   &    f_{12} \partial_{012} + f_{13} \partial_{013} + f_{23} \partial_{023}. 
\\
									&  -(f_1 \partial_{01} + f_2 \partial_{02} + f_3 \partial_{03})  \ar[ur]_{(d_1^2)^{i}}  &  
}
\end{align}
\end{center}

Recall that $\mathrm{rank}(\pi)=2$ and that the algebra of Casimirs is generated by $Q_1, Q_2$. Let $k_i=\dim\mathbb{R}_i[Q_1,Q_2]$ be the dimension of the space of homogeneous Casimirs functions of degree $i$, and set $r_i=\dim V_i$. 

Since $X_0=0$, it is $\mathrm{Im}(\di^0_i)\subset\mathrm{ker}((d^1_2)^i)$. Using the splitting (\ref{d_1}) of $\di^1_i$ and  (\ref{obs2}) one has that 
\[H^1_i(U_\Gamma,\pi)=\ker((d^1_1)^i)\oplus\left(\ker((d^1_2)^i) /\mathrm{Im} (\di^0_i)\right).\]
Let $\mathcal{A}=\mathbb{R}[x_1,x_2,x_3]$ and $\phi=\frac{1}{2}Q_2$. The restriction of $\pi$ on $\mathcal{A}$ is then determined by $\phi$ in the sense that $\{x_{\sigma(i)},x_{\sigma(j)}\}=\partial_{\sigma(k)}\phi$ for every cyclic permutation $\sigma$ of $(1,2,3)$.  Denote this Poisson algebra by $(\mathcal{A},\pi_\phi)$. 

The $\partial_j\phi$ have only one common zero at the origin, the vertex of the cone defined by $\phi=0$, and for this, the Milnor number of $\mathcal{A}/\langle\partial_1\phi,\partial_2\phi,\partial_3\phi\rangle$ is finite and equal to $1$. Fixing the weight vector $\bar{\omega}=(1,1,1)$, $\phi$ is then weight homogeneous of weight $\bar{\omega}(\phi)=\deg(\phi)=2$ and has an isolated singularity.

From (\ref{d^0}),(\ref{d^1}), summing over all polynomial degrees $i$, we get that
\[\bigoplus_i\ker((d^1_2)^i) /\mathrm{Im} (\di^0_i)=\mathbb{R}[x_0]\otimes H^1(\mathcal{A},\phi),\]
where the second term on the right side is the first formal Poisson cohomology group of $(\mathcal{A},\pi_\phi)$. Let $E_{\bar{\omega}}=\sum_{r=1}^3 x_r\partial_r$ be the (weighted by $\bar{\omega}$) Euler vector field on $\mathbb{R}^3$. Since $\bar{\omega}(\phi)\neq\mathrm{Div}(E_{\bar{\omega}})=3$, by \cite[Proposition 4.5]{P06}, it is $H^1(\mathcal{A},\phi)=\{0\}$. Hence,

\[H^1_\mathrm{formal}(U_{\Gamma}, \pi)=\bigoplus_i\ker((d^1_1)^i)=\big[\bigoplus_i\ker(\di^0_i)\big]\partial_0\]
and we get our claim for $H^1_\mathrm{formal}(U_{\Gamma}, \pi)$.

\noindent By the splitting of $\di^1_i$, $\di^2_i$, and (\ref{obs2}), (\ref{obs1}), it is
\[H^2_\mathrm{formal}(U_\Gamma, \pi)=\bigoplus_i\ker((d^2_2)^i)/\mathrm{Im} ((d^1_2)^i)=\mathbb{R}[x_0]\otimes H^2(\mathcal{A},\phi).\]
\noindent Hence, by \cite[Proposition 4.8]{P06} we get our claim for $H^2_\mathrm{formal}(U_{\Gamma}, \pi)$. 

It is easy to see checking (\ref{d^3}) directly that $\dim(\mathrm{Im}(\di^3_i))=r_i-k_i$. By the result for $H^2_\mathrm{formal}(U_{\Gamma}, \pi)$ one has that $\dim(\mathrm{Im}(\di^2_i))=3r_i$, for all $i$, which gives the claim for $H^3_\mathrm{formal}(U_{\Gamma}, \pi)$.
Finally, $\dim H^4_i(U_\Gamma,\pi)=k_i$, which is  equal to the dimension of $\mathbb{R}_i[Q_1,Q_2]\simeq\mathbb{R}_i[Q_1,Q_2]\partial_{1234}$.
\end{proof}

\end{document}